\documentclass[reqno, 11pt]{amsart}
\usepackage{color}
\usepackage{amsmath}
\usepackage{amsfonts}
\usepackage{amssymb}
\usepackage{amsthm}
\usepackage{newlfont}
\usepackage{graphicx}
\usepackage{float}

\usepackage{tikz}
\usetikzlibrary{matrix,shapes,arrows,positioning,chains}

\newtheorem{thm}{Theorem}[section]
\newtheorem{cor}[thm]{Corollary}
\newtheorem{lem}[thm]{Lemma}
\newtheorem{prop}[thm]{Proposition}
\theoremstyle{definition}
\newtheorem{defn}[thm]{Definition}
\theoremstyle{remark}
\newtheorem{rem}[thm]{Remark}
\newtheorem*{rem*}{Remark}
\numberwithin{equation}{section}

\newcommand{\ed}{\end {document}}

\newcommand{\ka}{\kappa}

\begin{document}

\title[Convergence of AC]{Sharp convergence to steady states of Allen-Cahn}

\author[D. Li]{ Dong Li}
\address{D. Li, Department of Mathematics, the Hong Kong University of Science \& Technology, Clear Water Bay, Kowloon, Hong Kong}
\email{mpdongli@gmail.com}

\author[C.Y. Quan]{Chaoyu Quan}	
\address{C.Y. Quan, SUSTech International Center for Mathematics, Southern University of Science and Technology,
	Shenzhen, P.R. China}
\email{quancy@sustech.edu.cn}

\author[T. Tang]{Tao Tang}
\address{T. Tang, Guangdong Provincial Key Laboratory of Computational Science and Material Design, Southern University of Science and Technology, Shenzhen, P.R. China; and Division of Science and Technology, BNU-HKBU United International College, Zhuhai, P.R. China}
\email{ttang@uic.edu.cn}

\author[W. Yang]{Wen Yang}
\address{\noindent W. Yang,~Wuhan Institute of Physics and Mathematics, Chinese Academy of Sciences, P.O. Box 71010, Wuhan 430071, P. R. China; Innovation Academy for Precision Measurement Science and Technology, Chinese Academy of Sciences, Wuhan 430071, P. R. China.}
\email{wyang@wipm.ac.cn}

\begin{abstract}
In our recent work 
 we  found a surprising breakdown of symmetry conservation: using standard numerical discretization with very high precision the computed numerical solutions corresponding to very nice initial data may converge  to completely incorrect steady
states due to the gradual accumulation of machine round-off error. We solved this issue by introducing a new
Fourier filter technique for solutions with certain band gap properties. To further investigate the attracting basin of steady states we classify in this work all possible bounded nontrivial steady states for the Allen-Cahn equation.
We characterize sharp dependence of nontrivial steady states on the diffusion coefficient and prove strict monotonicity of the associated energy. In particular,
we establish a certain self-replicating property amongst the hierarchy of steady states and give a full classification of their energies and profiles. We develop a new modulation theory and prove sharp convergence to the steady state with
explicit rates and profiles.
\end{abstract}

\maketitle

\vspace{1cm}
\section{Introduction}
In this paper, we consider the following one-dimensional Allen-Cahn equation posed on the periodic torus $\mathbb T=[-\pi,\pi]$:
\begin{equation}
\label{1.ac}
\begin{cases}
\partial_tu=\ka^2\partial_{xx}u-f(u),\\	
u\bigr|_{t=0}=u_0,	
\end{cases}
\end{equation}
where $\kappa>0$ measures the strength of diffusion,  $f(u)=u^3-u=F^{\prime}(u)$,
and $F(u)=(u^2-1)^2/4$ is the usual double-well potential. The function $u:\mathbb T\to \mathbb R$ represents the concentration difference of phases in an alloy and typically has values in the physical range $[-1,1]$.

In our concurrent work \cite{lqty2021}, we find a very  surprising breakdown of parity in typical
high-precision computation of \eqref{1.ac} with very smooth initial data.  For example take $\ka=1$ and consider the equation \eqref{1.ac} with the initial data $u_0(x)$ being an odd function of $x$
such as $u_0(x)=\sin x$. By simple PDE arguments the smooth solution should preserve the
odd symmetry for all time.  However numerical discretized solutions turn out  to fail to conserve
this parity and converge quickly to the spurious states $u=\pm 1$ in not very long time
simulations. This striking contradiction is a manifestation of the gradual accumulation of non-negligible
machine round off errors over time. To resolve this issue, we introduced a new Fourier filter method which works successfully for a class of initial  data with certain symmetry and band-gap properties. By eliminating the  unwanted projections into the unstable directions at each iteration, we rigorously  show that the filtered solution will converge to the true steady state in long time simulations.

A natural next task is to understand the situation for general solutions without symmetries or
band-gap properties.  The pivotal step is to categorize  the steady states of the elliptic Allen-Cahn equations and analyze in detail their spectral properties. For full generality we shall
consider the steady states of \eqref{1.ac} on the whole real axis, i.e.
\begin{equation}
\label{1.eac}
\kappa^2u''+u-u^3=0\quad \mathrm{in}\quad \mathbb{R}.
\end{equation}
In \cite{de1978} De Giorgi  raised the  problem about proving that
bounded solutions  to $\Delta u = F^{\prime}(u)$  in dimensions $2\le n\le 8$ which are monotone in one direction, must depend only on one variable in dimension. Since then there are many works in understanding the structure of the solutions. Particularly, in dimension $n=2$ and $n=3$, Ghoussoub-Gui \cite{gg1998} and Ambrosio-Cabr\'e \cite{ac2000} proved the conjecture respectively. Savin \cite{savin2009} proved the De Giorgi conjecture up to dimension 8 under some additional assumption. The same conclusion has been obtained by Wang \cite{wang2017} with a different method. In \cite{dkw2011} Del Pino-Kowalczyk-Wei  established the existence of a counterexample in dimensions $n\geq 9$. We refer the readers to \cite{aac2001} for  background on De Giorgi's conjecture, \cite{cs2014,cs2015} for the study on the symmetry properties of the solutions to the fractional Allen-Cahn equation and the recent survey \cite{cw2018} for some related open problems.  It is known that the monotone solutions to \eqref{1.eac} in any dimension are stable solutions, i.e., the second variation of the associated energy is non-negative, where the energy functional is defined as
\begin{equation}
\label{1.energy}
E(u)=\int_{\mathbb{R}}\left(\frac{\kappa^2}{2}|\nabla u|^2+\frac14(1-u^2)^2\right)dx.
\end{equation}
Recently, there is a new counterpart problem for the stable solutions to Allen-Cahn equation (see e.g. \cite{savin2009,pw2013,lww2017} and references therein). Based on the monotonicity assumption it is natural to consider the two sides limit (without loss of generality we assume the function is monotone in $x_n$)
\begin{equation}
	u^+:=\lim_{x_n\to+\infty}u,\quad u^-:=\lim_{x_n\to-\infty}u.
\end{equation}
It is known that the limit functions only depend on the previous $n-1$ variables. Savin \cite{savin2009} has proved if $u^\pm$ are $1-$D, then the original $u$ is also $1-$D. From a general perspective
it is of some importance to study the stable solutions and its energy functional \eqref{1.energy}
in order to understand the structure of steady states. In the first part of this paper, we shall classify all the steady states of the $2\pi$-periodic solutions to \eqref{1.eac}. Furthermore we consider the variation on the energy of the ground state with respect to $\kappa$.

\begin{thm}
\label{thm1.1}
Let $0<\kappa<1$ and $m_\kappa$ be the largest positive integer such that $m_\kappa \kappa<1$, then equation \eqref{1.eac} admits exactly $m_\kappa$ non-constant $2\pi$ periodic solutions up to some translation and odd reflection.	
\end{thm}	

It is known that any periodic solution of \eqref{1.eac} is bounded. By Modica's estimate one can get that $|u|\leq1$, see \cite{modica1985}. For convenience of the readers we shall include an elementary proof for the one dimensional case, see Proposition \ref{pra.1}. By using Proposition \ref{pra.1} it suffices
for us to consider periodic solutions to \eqref{1.eac} satisfying $|u|<1$, since $u\equiv1$ or $u\equiv-1$  provides the trivial global minimizers of \eqref{1.eac} from an energy perspective.

To state the next result, we define the energy for  $2\pi$-periodic functions $u\in H^1(\mathbb T)$:
\begin{equation}
\label{1.energy-2}
E_\kappa(u)=\int_{\mathbb{T}}\left(\frac{\kappa^2}{2}|\partial_xu|^2+\frac14(1-u^2)^2\right)dx.
\end{equation}
Define
\begin{equation}
\label{1.energy-quantity}
E_{\ka} = \inf_{u \in \mathcal S}E_\kappa(u),
\end{equation}
where
\begin{align}
\mathcal S=\{\phi~|
~\phi(x)\in H^1(\mathbb T)~\mbox{solves}~\eqref{1.eac},
~\phi(x)=\phi(x+2\pi)
~\mathrm{and}~|\phi|<1,~x\in\mathbb R\}.
\end{align}
For the $2\pi$-periodic solutions of \eqref{1.eac}, we call $u$ a ground state if $u$ is the least energy solution. For fixed $0<\kappa<1$, by using the classification result in Proposition \ref{pra.1}, we  can prove that the ground state solution is unique up to a translation and reflection.
To fix the symmetries it is convenient to introduce the notion of odd zero-up ground states
(see Definition \ref{def3.2}). In particular if a ground state solution $u$ is odd and satisfy
$u'(0)>0$, we shall call it an odd zero-up ground state and  denote it by $U_\ka$. For any $0<\ka<1$, define $m_{\ka}\ge 1$ as the unique integer such that
	\begin{align}
		\frac 1 {m_{\ka}+1} \le \ka  <\frac 1 {m_{\ka}}.
	\end{align}
	For each $j=1,\cdots,m_{\ka}$, define (note below that $j\ka<1$)
	\begin{align} \label{tmp_kj01}
		\tilde u_{\ka, j} (x) = U_{j\ka} ( jx).
	\end{align}
	Then $\{ \tilde u_{\ka, j} \}_{j=1}^{m_{\ka}}$ are all the possible odd
	zero-up solutions to \eqref{3.ac}. Furthermore the energies of $\tilde u_{\ka, j}$ are
	given by
	\begin{align}
		E_{\ka}(\tilde u_{\ka, j} ) & = \int_{\mathbb T}
		\Bigl( \frac 12 (\ka \partial_x {\tilde u_{\ka,j} } )^2
		+ \frac 14 ( \tilde u_{\ka,j}^2-1)^2 \Bigr) dx
		= E_{j\ka} (U_{j\kappa}).
	\end{align}

With this notation, we now state the following structure
theorem on the energy functional $E_\kappa(u)$ of the $2\pi$-periodic solutions.

\begin{thm}
	\label{thm1.2}
	Let $E_\kappa$ be defined in \eqref{1.energy-quantity}. Then it can be achieved for any $\kappa>0.$ In addition, we have
	\begin{itemize}
		\item [(a).] $E_{\ka} =\frac \pi2$ for $\ka\ge 1$ and it is only achieved by the zero function;
		\item [(b).] $E_{\ka}$ is achieved by $U_\ka$ whenever $\kappa\in(0,1)$;
		\item [(c).] If $0<\ka_1<\ka_2\le 1$, then there is strict monotonicity $E_{\ka_1} < E_{\ka_2}$;
		\item [(d).] The odd zero-up ground state $U_\kappa$ satisfies
		\begin{align}
		\left|U_\kappa(x)-\tanh\left(\frac{x}{\sqrt{2}\kappa}\right)\right|\leq C\exp\left(-\frac{d}{\kappa}\right),
		\end{align}
		for some universal positive constants $C$ and $d$, and
		\begin{align}
		\lim\limits
		_{\ka \to 0} \frac {E_{\ka}} { {\ka} } = \frac 4 3 \sqrt 2 >0.
		\end{align}
		\item [(e).] For $0<\kappa<1$, the $2\pi$-periodic solutions of problem \ref{1.eac} have the following replica property: any $2\pi$-periodic solution $u$ of \eqref{1.eac} which is not
		identically $\pm 1$ or $0$ must  coincide (after some shift and odd reflection
		if necessary) with $\tilde u_{\kappa, j}$ for
		some integer $j <1/\kappa$.  Here $\tilde u_{\kappa,j}$ is defined in
		\eqref{tmp_kj01}.
	        Furthermore $E_\ka(u)=E_{m\ka}(U_{m\ka}).$
	\end{itemize}
\end{thm}

From Theorem \ref{thm1.1} we can see that $0$ is the only $2\pi$-periodic solution to equation \eqref{1.eac} whenever $\ka\ge1$. This is in complete accord with numerical experiments. Furthermore when $\ka>1$, $u(x,t)$ converges to $0$ exponentially as $O(e^{-(\ka^2-1)t})$, while
the convergence rate becomes $O(t^{-\frac12})$ for $\kappa=1$. In the second part of this work, we
shall rigorously prove these convergence results and identify the explicit profiles. Our strategy is quite robust and we shall illustrate it for  a  general fractional Allen-Cahn equation
\begin{equation}
\label{1.fac}	
\begin{cases}
\partial_tu=-\ka^2\Lambda^\gamma u+u-u^3,\quad (x,t)\in\mathbb{T}\times (0,\infty),\\	
u\bigr|_{t=0}=u_0,
\end{cases}		
\end{equation}
where $\Lambda^\gamma=(-\partial_{xx})^{\frac{\gamma}{2}}$ is the fractional Laplacian of order $\gamma\in(0,2].$ When $\gamma=2$ it coincides with the usual $-\partial_{xx}.$
For simplicity of presentation we state below a simple version
of the obtained results in Section 4.  Sharper results concerning profiles, rates etc can be
found in Section 4.
\begin{thm}[Vanishing as $t \to \infty$]
\label{thm1.3}	
Let $\ka\ge 1$ and $0<\gamma\le 2$. Assume $u_0$ is $2\pi$ periodic, odd and bounded. Suppose $u$ is the solution to \eqref{1.fac} corresponding to the initial data $u_0$.
If $\ka >1$, we have
\begin{equation}
\label{1.con}
\begin{aligned}
u(x,t)=e^{-(\ka^2-1)t}\alpha_*\sin x+r(t),\qquad \forall\, t\ge1,
\end{aligned}
\end{equation}
where $\alpha_*$ depends on $u_0,\gamma$ and $\ka,$ and $\|r(t)\|_{H^{10}(\mathbb T)}=o(e^{-(\ka^2-1)t})$ as $t\to+\infty$.

For $\ka=1$, we have
\begin{align}
\label{1.con1}
u(x,t)=t^{-\frac12}\beta_*\sin x+r_1(t),\qquad \forall\, t\ge1,
\end{align}
where $\beta_*$ depends on $u_0,\gamma$ and $\|r_1(t)\|_{H^{10}(\mathbb T)}=o(t^{-\frac12})$ as $t\to+\infty$.
\end{thm}

When $\ka\in(0,1)$, the corresponding theory of convergence becomes quite involved. Indeed, from Theorem \ref{thm1.1} we see that the number of steady states (up to identification of
symmetry) increases as $O(1/\ka)$ when $\ka$ decays to zero.
At the moment there is no general theory for the precise identification of
the corresponding steady for arbitrary initial data.
However for a class of benign initial data, we have  the following 
precise and definite convergence results.

\begin{thm}
\label{thm1.4}	
Let $0<\ka<1$. Assume the initial data $u_0:\mathbb{R}\to\mathbb{R}$ is $2\pi$-periodic, odd and  non-negative in $[0,\pi],$ then we have $u(x,t)\to U_\ka$ or $0$ as $t\to\infty$. Moreover, if $u_0\neq0$ and
and $E_\ka(u_0)\leq\frac\pi2$, then $u(x,t)\to U_\ka$ as $t\to\infty$ and the rate of convergence is exponential in time.
%
%
\end{thm}

To our best knowledge Theorem \ref{thm1.4} along with earlier results are the first sharp
quantitative convergence results on the Allen-Cahn equation. We plan to develop this
program on much more general phase field models in forthcoming works.

The rest of this paper is organized as follows. In Section 2 we introduce preliminary analysis of the steady states and examine in detail the profiles of the ground states. In Section 3
we prove Theorems \ref{thm1.1} and \ref{thm1.2}, give full classification of the steady states and analyze their profiles and energy monotonicity.
In Section 4 we study the convergence of the general parabolic Allen-Cahn equation \eqref{1.fac}, and prove Theorems \ref{thm1.3} and \ref{thm1.4}. In Section 5 we give concluding remarks.
The proof of Proposition 2.2 is given in the appendix.

\section{Classification of steady states}\label{sect2}
The solutions to $\kappa^2 u^{\prime\prime} +u-u^3=0$ are remarkably rigid, as documented
by the following ``patching" of nonlinear solutions.

\begin{prop}[Patching of nonlinear solutions via reflection]
\label{pr-sect2}
The following hold.
\begin{itemize}
\item \underline{Even reflection}. Suppose $\kappa>0$, and for some $\epsilon_0>0$ we have
\begin{align}
\kappa^2 u^{\prime\prime} +u -u^3 =0, \qquad \forall\, -\epsilon_0<x<0,
\end{align}
where $u \in C^2( (-\epsilon_0, 0) )$ and we assume $\lim\limits_{x\to 0-} u^{\prime}(x)=0$.
Define $u(x) =u(-x)$ for $0<x<\epsilon_0$. Then it holds that $u \in C^{\infty}((-\epsilon_0,\epsilon_0) )$ with $u^{\prime}(0)=0$ and solving the same equation on the
whole interval.

\item \underline{Odd reflection}. Suppose $\kappa>0$, and for some $\epsilon_0>0$ we have
\begin{align}
\kappa^2 u^{\prime\prime} +u -u^3 =0, \qquad \forall\,  0<x<\epsilon_0.
\end{align}
where $u \in C^2( (0, \epsilon_0) )$ and we assume $\lim\limits_{x\to 0+} u(x)=0$.
Define $u(x) =-u(-x)$ for $-\epsilon_0<x<0$. Then it holds that $u \in C^{\infty}((-\epsilon_0,\epsilon_0) )$ with $u(0)=0$ and solving the same equation on the
whole interval.
\end{itemize}
\end{prop}

\begin{figure}[!h]
\centering
\includegraphics[trim={1in 0.5in 0.6in 0.3in},clip,width=.9\textwidth]{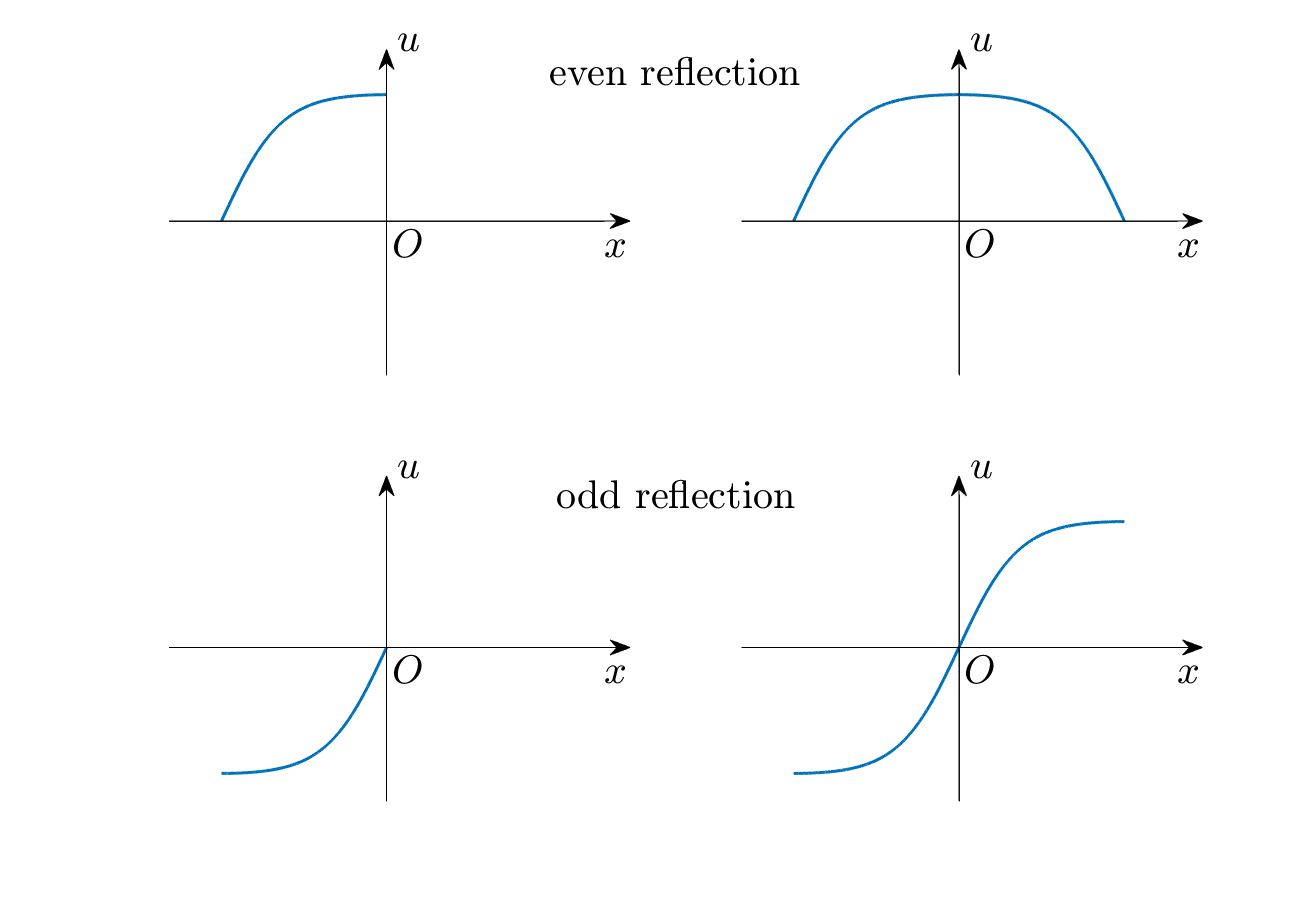}
\caption{Even reflection and odd reflection.}\label{fig:oddeven}
\end{figure}

\begin{rem}
\label{re-2.1}
{ Proposition \ref{pr-sect2} shows that the solution is remarkably rigid. If we know the profile of $u$ on some interval $(a,b)$ with $u(a)=0,~u'(b)=0$. Then the  solution can be uniquely determined on
a larger interval.}
\end{rem}

\begin{proof}
We shall only prove the first case as the second case is similar. First it is not difficult to that $u$ has
bounded derivatives in $[-\epsilon_0/2, 0)$ which can be extended to $0$ from the left. The
extended $u$ satisfies the equation on $(-\epsilon_0, 0) \cap (0,\epsilon_0)$. Furthermore the
equation also holds at $x=0$ up to third order derivatives. Then we can bootstrap the regularity
of $u$ by using the equation and conclude that $u \in C^{\infty}$.
\end{proof}

Observe that for $u_0(x) =\sin x$, since $f(u)=u^3- u$,  we have
\begin{align}
u(x,t) =\sum_{\text{$m\ge 1$: {$m$ is odd}}} c_m(t) \sin m x.
\end{align}
In particular it follows that the corresponding steady state $u_{\infty}$ is odd.
If $2\pi$ is the minimal period (such solution is actually the odd zero-up ground state up to a reflection if necessary, see Definition \ref{def3.2}),
then $u_{\infty}(0) = u_{\infty}^{\prime}(\frac {\pi} 2)=0.$
In addition, $u_\infty$ satisfies the steady state equation
\begin{align} \label{2.steady}
\ka^2 u^{\prime\prime} - f (u) =0\quad\mbox{on}\quad  \mathbb T.
\end{align}
We may look for the steady state such that it is monotonically increasing
on $[0, \frac {\pi }2]$ with $u_{\infty}(0) = u_{\infty}^{\prime}(\frac {\pi} 2)=0$.  Effectively
by using reflection symmetry, the whole graph of $u_{\infty}$ will be determined by its
graph on the interval $[0,\frac {\pi}2]$.

To simplify the notation we now write $u= u_{\infty}$ as the desired steady state.
We consider the regime $0<\ka <1$ (for simplicity we
suppress the notational dependence of $u$ on $\ka$). Denote $u(\frac {\pi}2) = N<1$ and observe that we should have $N\to 1$ as $\ka \to 0$.
Multiplying $\eqref{2.steady}$ by $u'$ and using $u^{\prime}(\frac {\pi} 2) =0$, we have
\begin{align}
(u^{\prime})^2 = \frac 1 { 2 {\ka^2} }
( (u^2-1)^2 - (N^2-1)^2).
\end{align}
If $u$ is monotonically increasing, it satisfies
\begin{align}
u^{\prime}(x) = \frac 1 {\sqrt 2 {\ka} }
\sqrt{ (u^2-1)^2 - (N^2-1)^2},
\end{align}
with $u(0)=0$, $u(\frac {\pi} 2) =N$.
We then obtain
\begin{align}
\int_0^N \frac 1 { \sqrt{ (u^2-1)^2 - (1- N^2)^2} } du = \frac {\pi} {2\sqrt 2 \ka}.
\end{align}
For each fixed $0<\ka < 1$, there exists a unique $0<N=N(\ka) <1$ such that the above
identity holds. Furthermore one can determine the dependence of $N$ on $\ka$. Indeed
by a change of variable $u \to N \sin \theta $, the left-hand side of the above equation is denoted by
\begin{equation}\label{def:gN}
\begin{aligned}
g(N) &:= \int_0^N \frac 1 { \sqrt{ (u^2-1)^2 - (1- N^2)^2} } du  = \int_0^{\frac{\pi}2} \frac 1 {\sqrt{ 2 -N^2(1+\sin^2 \theta) } } d\theta.
\end{aligned}
\end{equation}
Here we note that $g$ is monotonically increasing on $[0,1)$, $g(0) = \frac {\pi} {2\sqrt 2}$ and $g(1) =\infty$.
In particular we see the necessity of $\ka <1$!
Otherwise if $\ka\geq 1$, the equation \eqref{2.steady} admits the trivial solution $u\equiv0$.
For $0<\ka \ll 1$, it is not difficult to check that
\begin{align}
1-N(\ka)=  O(e^{-\frac c  {\ka} } ).
\end{align}
Sharper asymptotics can certainly be derived.

We summarize the above discussion as the following proposition.
\begin{prop}[Characterization of a special steady state for $0<\ka<1$] \label{sdy1}
The following hold:

\begin{enumerate}
\item
The function $g$ defined in \eqref{def:gN} is monotonically increasing on $[0,1)$, $g (0) = \frac {\pi} {2\sqrt 2}$ and $g(N) \to\infty$ as $N\to 1$.

\item For any $0<\ka <1$, there exists a unique $0<N_{\ka} <1$ such that
\begin{align}\label{eq:gN_ka}
g(N_\ka) = \frac {\pi} {2\sqrt 2 \ka}.
\end{align}
Furthermore we have
\begin{align} \label{sdy1.200a}
 c_1 e^{-\frac {c_2} {\ka}} <1- N_{\ka} < c_3 e^{-\frac {c_4} {\ka} },
 \end{align}
 where $c_i>0$, $i=1,\cdots,4$ are absolute constants.

\item For any $0<\ka<1$, there exists a $2\pi$-periodic $C^{\infty}$ odd function $u_{\ka}$ such that
\begin{itemize}
\item $u_{\ka}$ is a steady state, i.e. $\ka^2 u_{\ka}^{\prime\prime}
-f(u_{\ka}) =0$.
\item $u_{\ka}(0)=u_{\ka}^{\prime}(\frac {\pi}2)=0$, $u_{\ka}(\frac {\pi}2)=N_{\ka}$, and $u_{\ka}$ is monotonically increasing on $[0,\frac {\pi}2]$.
\item $u_{\ka}(\pi-x) =u_{\ka}(x)$ for $\frac {\pi}2 \le x \le \pi$.
\end{itemize}
Moreover for $0<\ka\ll 1$, we have
\begin{align} \label{sdy1.2}
0\le  \tanh\left( \frac x {\sqrt 2 \ka } \right)- u_{\ka}(x)\le  \exp(-\frac {c_5} {\ka} ),
\qquad\forall\, 0\le x\le \frac {\pi}2,
\end{align}
where $c_5>0$ is an absolute constant.
\end{enumerate}
\end{prop}
\begin{proof}
1. For fixed $N\in[0,1)$, taking the derivative of $g(N)$ with respect to $N$, we have
\begin{equation}
g'(N)=\int_0^{\frac\pi2}\frac{N(1+\sin^2\theta)}{(2-N^2(1+\sin^2\theta))^{\frac32}}d\theta.
\end{equation}	
It is not difficult to verify that both the numerator and denomenator are positive in $(0,\frac\pi2)$, therefore, we have shown that $g'(N)>0$ and it proves that $g$ is monotonically increasing for $N\in[0,1)$. When $N=0$, we have $g(0)=\int_0^{\frac\pi2}\frac{1}{\sqrt{2}}d\theta=\frac{\pi}{2\sqrt{2}}$. While as $N$ is close to $1$, we have
\begin{equation}
\label{ineq:gN}
\begin{aligned}
g(N) & = \int_0^{\frac \pi 2} \frac{1}{\sqrt{2-2N^2+N^2\cos^2\theta}} d\theta= \int_0^{\frac \pi 2} \frac{1}{\sqrt{2-2N^2+N^2\sin^2\theta}} d\theta \\
& > \int_0^1 \frac{1}{\sqrt{2-2N^2+N^2\theta^2}} d\theta
> \frac{1}{N}\int_\delta^1 \frac{1}{\sqrt{\delta^2 + \theta^2}} d\theta \\
&> \frac{1}{N}\int_\delta^1 \frac{1}{\sqrt{2}\theta} d\theta = -\frac{1}{\sqrt{2}N}\log \delta,
\end{aligned}	
\end{equation}
where
\begin{equation}
\delta = \sqrt{\frac{2-2N^2}{N^2}}\to 0,\quad \mbox{as } N\to 1.
\end{equation}
Here, it is assumed implicitly that $N> \sqrt{\frac 2 3}$ so that $\delta <1$.
Then, it is easy to see that $g(N)\to \infty$ as $N\to 1$.
\medskip

2.  The existence and uniqueness of $N_{\ka}$ follows easily from the behavior of the function
$g(\cdot)$.
Now we shall show the upper and lower bounds on $N_{\ka}$. It is clear from Step 1 that $N_{\ka}\to 1$ as $\ka\to 0$. If $N_{\ka}\le \sqrt{\frac 23}$, then $\ka$ is
bounded away from zero by an absolute constant and the desired estimate
\eqref{sdy1.200a} clearly holds in this case. Thus we only need to consider the
situation $N_{\ka}>\sqrt{\frac 23}$. To ease the notation
we denote $N=N_{\ka}$ with $N>\sqrt{\frac 2 3}$.
From \eqref{ineq:gN}, we have
\begin{equation}
\frac{\pi}{2 \sqrt 2 \ka} = g(N) > -\frac{1}{\sqrt{2}N}\log \delta,
\end{equation}
which yields directly
\begin{equation}
1-N> \frac{N^2}{2+2N} e^{-\frac{N\pi}{\ka}}.
\end{equation}
On the other hand, using $\sin\theta>\frac 2 \pi \theta$, we have
\begin{equation}
\begin{aligned}
\frac{\pi}{2 \sqrt 2 \ka} & = g(N) = \int_0^{\frac \pi 2} \frac{1}{\sqrt{2-2N^2+N^2\sin^2\theta}} d\theta < \int_0^{\frac \pi 2} \frac{1}{\sqrt{2-2N^2+N^2\left(\frac 2 \pi\theta\right)^2}} d\theta \\
& = \int_0^{1} \frac{\pi}{2N\sqrt{\delta^2+\tilde\theta^2}} d\tilde\theta = \int_0^{\delta} \frac{\pi}{2N\sqrt{\delta^2+\tilde\theta^2}} d\tilde\theta+  \int_\delta^{1} \frac{\pi}{2N\sqrt{\delta^2+\tilde\theta^2}} d\tilde\theta\\
&< \frac{\pi}{2N} - \frac{\pi}{2\sqrt{2}N} \log\delta,
\end{aligned}	
\end{equation}
which yields
\begin{equation}
1-N < \frac{N^2}{2+2N} e^{2\sqrt 2-\frac{2N}{\ka}}.
\end{equation}
Therefore, when $N> \sqrt{\frac 2 3}$, we have
\begin{equation}\label{ineq:N}
\frac 1 {3 + \sqrt 6} e^{-\frac \pi \ka} <\frac{N^2}{2+2N} e^{-\frac{N\pi}{\ka}}< 1-N < \frac{N^2}{2+2N} e^{2\sqrt 2-\frac{2N}{\ka}} < \frac{e^{2\sqrt 2}}{4} e^{-\frac{2\sqrt 2}{\sqrt 3\ka}}.
\end{equation}

3. Fix $0<\ka <1$ and  consider
the function
\begin{align}
h(u)= \int_0^u  \frac {\sqrt 2 {\ka} } {\sqrt{ (y^2-1)^2- (1-N^2 )^2}}dy,
\quad 0\le u\le N.
\end{align}
Clearly $h:\, [0,N ]\to [0,\frac {\pi}2]$ is strictly monotonically increasing and bijective.
The inverse map of $h(u)$ then defines the desired function $u_{\ka}$ on the interval $[0,\frac {\pi}2]$. {It is known that
$u(\frac {\pi}2) =N $  and $u^{\prime}(\frac {\pi} 2) =0$, by Proposition \ref{pr-sect2} we derive that $u$ is even respect to $x=\frac\pi2,$ i.e., $u(x)=u(\pi-x)$ for $x\in\left(0,\frac\pi2\right).$}

Finally to show \eqref{sdy1.2}, we denote $\theta(x) = \tanh(\frac x {\sqrt 2 \ka} )$. Clearly
\begin{equation}
\left\{
\begin{aligned}
&\frac {du_{\ka} } {dx} = \frac 1 {\sqrt 2 \ka}
\sqrt{ (u^2_{\ka}-1)^2 - (N^2-1)^2},   \\
&\frac {d\theta} {dx} =  \frac 1 {\sqrt 2 \ka} (1-\theta^2),  \\
&u_{\ka}(0) = \theta(0)=0,
\end{aligned}
\right.
\end{equation}
Observe that
\begin{align}
\frac {du_{\ka} } {dx} < \frac 1 {\sqrt 2 \ka} (1-u_{\ka}^2).
\end{align}
Denote $\eta(x) = u_{\ka}(x) - \theta(x)$. Clearly $\eta(0)=0$ and
\begin{align}
\eta^{\prime} < -\frac 1 {\sqrt 2 \ka} (u_{\ka}+\theta) \eta.
\end{align}
{This implies that $\eta(x) \le 0$ for all $0\le x \le \frac {\pi}2$.}
Thus
\begin{align}
u_{\ka} (x) -\theta(x) \le 0, \qquad \forall\, 0\le x \le \frac {\pi}2.
\end{align}
We now show the lower bound. By \eqref{sdy1.200a} we can choose
an absolute constant $\delta_0>0$ sufficiently small such that
\begin{align}
1-N^2 <c_3 e^{-\frac {c_4} {\ka}}< \frac 1 {100} (1-\theta(x)^2),
\qquad\forall\, 0\le x \le \delta_0.
\end{align}
This implies
\begin{align}
\sqrt{ (1-u_\ka(x)^2)^2 - (1-N^2)^2 }  +1-\theta(x)^2 \ge
C ( 1-u_\ka(x)^2 + 1- \theta(x)^2),
\end{align}
for $x\in[0,\delta_0]$ and some generic constant $C>0$. Now observe that
\begin{align}
\frac d {dx} \eta = \frac 1 {\sqrt 2 \ka}
 \frac {(1-u_\ka^2)^2 - (1-\theta^2)^2 - ( 1-N^2)^2}
{ \sqrt{(1-u_\ka^2)^2-(1-N^2)^2} + 1-\theta^2}.
\end{align}

Thus for $0<x\le \delta_0$, we have
\begin{align}
\frac {d\eta} {dx} =O(\ka^{-1}) \cdot \eta + O(e^{-\frac {b_0} {\ka} }),
\end{align}
where $b_0>0$ is an absolute constant. Then for $0<x \le b_1$ ($b_1>0$ is
a sufficiently small absolute constant), we have
\begin{align}
\sup_{0\le x \le b_1}| \eta(x)| \le e^{-\frac {b_2} {\ka} },
\end{align}
where $b_2>0$ is an absolute constant.  Thus for $0\le x \le b_1$, we have
\begin{align}
u_{\ka}(x) - \theta(x) \ge -e^{-\frac {b_2} {\ka} }.
\end{align}
For $b_1 \le x \le \frac {\pi}2$, by using monotonicity, we have
\begin{align}
u_{\ka}(x) - \theta(x) & \ge u_{\ka}(b_1) - \theta(\frac {\pi}2)
=u_{\ka}(b_1)-\theta(b_1) + \theta(b_1) -\theta(\frac {\pi}2) \notag \\
& \ge  -e^{-\frac {\tilde c_5} {\ka}},
\end{align}
where $\tilde c_5$ is an absolute constant. The desired result then follows easily by
collecting the estimates.
\end{proof}

A periodic steady state on $\mathbb T$ can be extended naturally to the whole space $\mathbb R$.
Therefore, such periodic steady states can be seen as a special solution to the following steady state equation defined in $\mathbb R$
\begin{equation}
\label{a.ac}
\ka^2 u''+u-u^3=0,\quad x\in\mathbb{R}.
\end{equation}
Multiplying $\eqref{a.ac}$ by $u'$, we derive that
\begin{equation}
\label{a.id}
\ka^2 (u')^2+u^2-\frac12u^4~\mbox{is a constant, denoted by}~C,
\end{equation}
which can be rewritten as
\begin{equation}
\label{a.id1}
\ka^2(u')^2=\frac12(u^2-1)^2+C-\frac12.
\end{equation}
Concerning the solution of \eqref{a.ac}, we have the following  result.
\begin{prop}
\label{pra.1}
Let $u$	be a bounded solution to \eqref{a.ac} and $C$ be the constant defined in \eqref{a.id}, then the following hold.
\begin{enumerate}
	\item [(1)] If $C>\frac12$, $u$ can not exist.
	\item [(2)] If $C=\frac12$, then $u= \pm\tanh\frac{x+c}{\sqrt{2}\ka}$ or $u\equiv \pm 1$.
	\item [(3)] If $0<C<\frac12$, then $u$ is a periodic function and $|u|<1$.
	\item [(4)] If $C=0$, then $u\equiv0.$
\end{enumerate}	
\end{prop}
\begin{proof}
Although the above conclusion is a folklore, we provide the proof for the sake of
completeness in Appendix \ref{append1} and a graphical illustration in Figure \ref{fig:class}.
In the case of $0<C<\frac 1 2$, an odd periodic steady state which has its period precisely given
by $2\pi$ is characterized  by Proposition \ref{sdy1}.
\end{proof}

\begin{figure}[!h]
\centering
\includegraphics[trim={1in 0.5in 0.6in 0.5in},clip,width=1\textwidth]{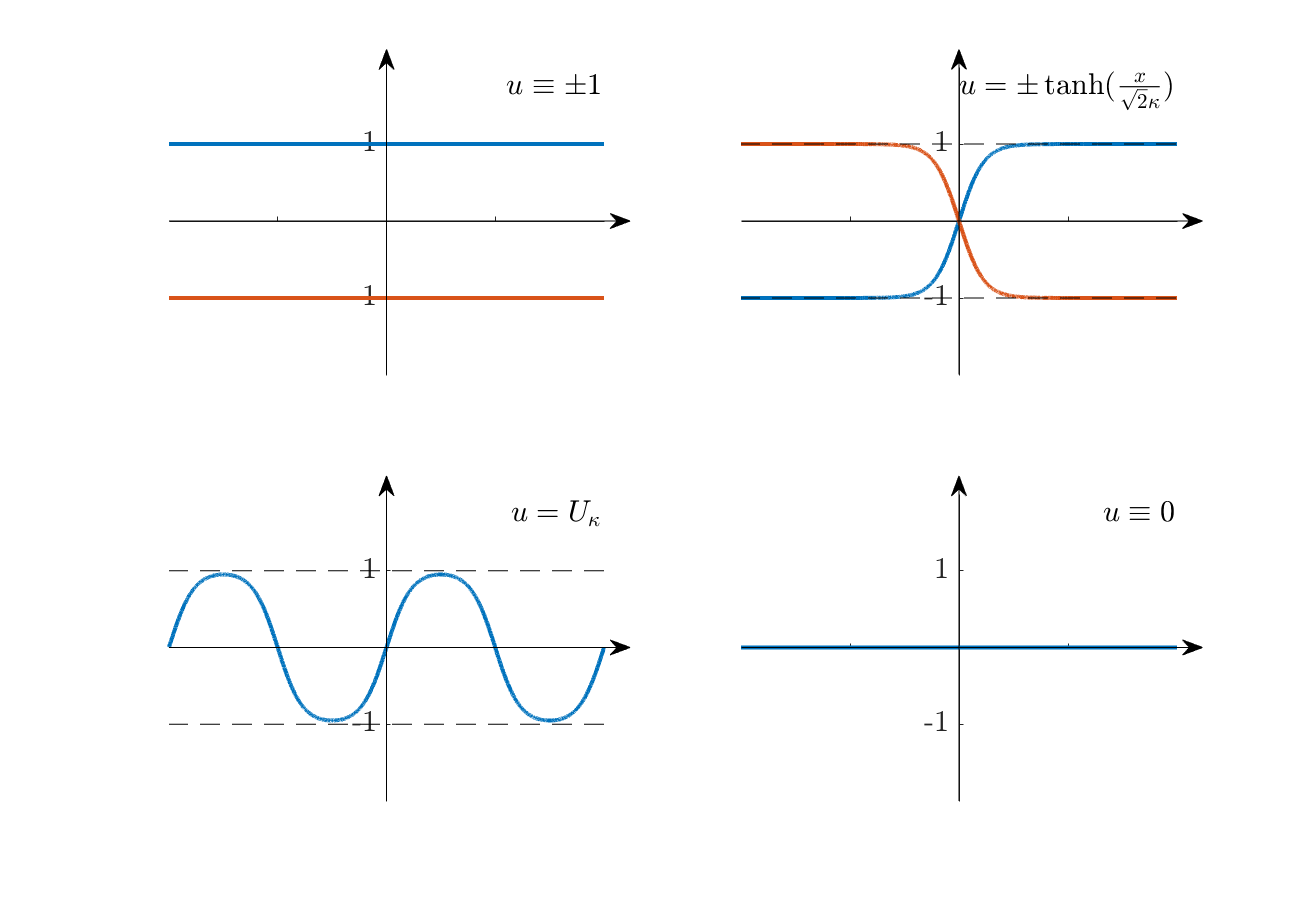}
\caption{Classification of bounded steady states of Allen-Cahn ($\ka = \frac12$).}\label{fig:class}
\end{figure}

\section{Classification of steady state energy}
In this section, we consider the energy $E_\ka(u)$ (see \eqref{1.energy-2}) of solutions to \eqref{1.eac}. From the discussion of Section 2, we see that a nontrivial bounded steady state is a $2\pi$-periodic function. Therefore, we focus on the following problem
\begin{equation}
\label{3.ac}	
\kappa^2u''+u-u^3=0,\quad u(x)=u(x+2\pi)~\mbox{for}~x\in\mathbb{R}.
\end{equation}
For simplicity of presentation, we introduce the following definition.
\begin{defn}[Odd zero-up solution]
\label{def3.1}
We shall say that $u$ is an odd zero-up solution to \eqref{3.ac} if the solution $u$ is odd  and $u^{\prime}(0)>0$.
\end{defn}

\begin{defn}[Odd zero-up ground states]
\label{def3.2}
For each $0<\ka <1$, we define $U_{\ka}= u_{\ka}$, where $u_{\ka}$
is obtained in Proposition \ref{sdy1} as the odd zero-up ground state solution
to \eqref{3.ac}. We also define the odd zero-up ground state energies $E_{\ka}^{(0)}$ as
\begin{equation}
\label{E_grd001}
\begin{aligned}
E_{\ka}^{(0)} & = \int_{\mathbb T}
\Bigl( \frac 12 \ka^2 ( U_{\ka}^{\prime}(x) )^2 + \frac 14 ( U_{\ka}(x)^2-1)^2  \Bigr)
dx\\&
= \int_{\mathbb T} \Bigl( \frac 12 (U_{\ka}(x)^2-1)^2 -\frac 14
(N_{\ka}^2-1)^2 \Bigr) dx,
\end{aligned}
\end{equation}
where we recall $0<N_{\ka}<1$ is the unique number satisfying
\begin{equation}
\label{3.def-N}
\int_0^{\frac\pi2}\frac{1}{\sqrt{2-N_{\ka}^2(1+\sin^2\theta)}}d\theta=\frac{\pi}{2\sqrt{2}\ka}.
\end{equation}
\end{defn}

For any solution of \eqref{3.ac}, we assume that its minimal period is $2\pi/m$ for some suitable positive integer $m$. From the proof of Proposition \ref{pra.1}, it is not difficult to see that $u(x)$ has $2m$ zero points in $[x,x+2\pi)$ for any $x\in\mathbb{R}$ and $u(x)$ has odd symmetry with respect to any zero point. In addition, we can easily prove that the distance between any two consecutive zero points is same and equals to $2\pi/m$. After a suitable shift of the solution, we may assume $u(0)=0$ and $u$ is odd. On the other hand, if $u(x)$ is a solution to \eqref{3.ac}, obviously $u(-x)$ is also a solution. Hence, we may assume that $u'(0)>0$ after reflection if necessary. Therefore, in this section we shall restrict our discussion on the odd zero-up solutions of equation \eqref{3.ac}. Concerning all the odd zero-up solutions to \eqref{3.ac}, we have the following classification result

\begin{thm}[Classification of odd zero-up solutions to \eqref{3.ac}]
\label{tha.1}
	For any $0<\ka<1$, define $m_{\ka}\ge 1$ as the unique integer such that
	\begin{align}
		\frac 1 {m_{\ka}+1} \le \ka  <\frac 1 {m_{\ka}}.
	\end{align}
	Then there are only $m_{\ka}$ odd zero-up solutions to \eqref{3.ac}. More precisely
	the following hold:
	
	For each $j=1,\cdots,m_{\ka}$, define (note below that $j\ka<1$)
	\begin{align}
		\tilde u_{\ka, j} (x) = U_{j\ka} ( jx).
	\end{align}
	Then $\{ \tilde u_{\ka, j} \}_{j=1}^{m_{\ka}}$ are all the possible odd
	zero-up solutions to \eqref{3.ac}. Furthermore the energies of $\tilde u_{\ka, j}$ are
	given by
	\begin{align}
		E_{\ka, j} & = \int_{\mathbb T}
		\Bigl( \frac 12 (\ka \partial_x {\tilde u_{\ka,j} } )^2
		+ \frac 14 ( \tilde u_{\ka,j}^2-1)^2 \Bigr) dx
		= E_{j\ka}^{(0)}, \label{E_jeps002}
	\end{align}
	where $E_{j\ka}^{(0)}$ was defined in \eqref{E_grd001}.
\end{thm}
\begin{proof}
	Suppose $u$ is a possible odd zero-up solution to \eqref{3.ac}.
	The crucial observation is that we must have $u$ achieves its first peak
	at $x=\frac {\pi}{2j}$ for some integer $j\ge 1$.  Now make a change of variable
	$y= j x$, and $\tilde u(y) = u(x)$. Then clearly
	\begin{align}
		j^2 \ka^2 \frac {d^2}{dy^2} \tilde u -\tilde u + \tilde u^3=0,
	\end{align}
	$\tilde u(0)=0$, $\tilde u^{\prime}(0)>0$, and $\tilde u^{\prime}(\frac {\pi}2)=0$. From the proof in Step 3 of Proposition \ref{sdy1}, there exists a unique solution $u$ with $|u|<1$ solving the equation
	\begin{align}
		u^{\prime} = \frac 1 {\sqrt 2 \ka} \sqrt{ (1-u^2)^2- (1-N_{\ka}^2)^2},
	\end{align}
	with $u(0)=0$, $u^{\prime}(\frac {\pi}2)=0$.  As a consequence, we obtain that
	$\tilde u = U_{j\ka}$.  Now note that $j\ka<1$ and this gives the constraint
	$j\le m_{\ka}$.   The characterization \eqref{E_jeps002} follows from
	the fact that
	\begin{align}
		E_{\ka,j} &= \int_{\mathbb T} \Bigl( \frac 12 (\tilde u_{\ka,j}(x)^2-1)^2 -\frac 14
		(N_{j\ka}^2-1)^2 \Bigr) dx \notag \\
		&= \int_{\mathbb T} \Bigl( \frac 12 (U_{j\ka}(jx)^2-1)^2 -\frac 14
		(N_{j\ka}^2-1)^2 \Bigr) dx,
	\end{align}
	and the fact that $U_{j\ka}$ is $2\pi$-periodic.
\end{proof}

By Theorem \ref{tha.1} one can easily get Theorem \ref{thm1.1}. We notice that the $C^0$ estimate in the point (d) of Theorem \ref{thm1.2} follows easily by \eqref{sdy1.2}. While for the point (e), one can easily prove it by some direct computations. For the left conclusions in Theorem \ref{thm1.2}, we rephrase it as the following result for the odd zero-up solutions

\begin{thm}
[Monotonicity and asymptotics of odd zero-up ground state energies]
\label{grd_mono001}
For any $\ka>0$, define
\begin{align}
\tilde E_{\ka} = \inf_{u \in \mathcal S_O}
\int_{\mathbb T} \Bigl( \frac 12 (\ka \partial_x u)^2 + \frac 14 (u^2-1)^2 \Bigr)dx,
\end{align}
where
\begin{align}
\mathcal S_O=\{  \phi\mid \text{$\phi:\; \mathbb T \to \mathbb R$ is odd and $C^1$},~\phi'(0)>0 \}.
\end{align}
Then we have
\begin{itemize}
\item [(a).] $\tilde E_{\ka} =\frac \pi2$ for $\ka\ge 1$. Furthermore
\begin{align}
\int_{\mathbb T} \Bigl( \frac 12 (\ka \partial_x u)^2 + \frac 14 (u^2-1)^2 \Bigr)dx
> \frac {\pi} 2
\end{align}
for any $u$ not identically zero.

\item [(b).] $\tilde E_{\ka} = E_{\ka}^{(0)}$ for $0<\ka<1$. Moreover the infimum is only achieved by $U_{\ka}$.
\item [(c).] If $0<\ka_1<\ka_2\le 1$, then $\tilde E_{\ka_1} <\tilde E_{\ka_2}$.
\end{itemize}
Furthermore
\begin{align} \label{E_eplimit00}
\lim_{\ka \to 0} \frac {E_{\ka}^{(0)} } { {\ka} } = \gamma_* = \frac 4 3 \sqrt 2 >0.
\end{align}
\end{thm}

Before proving Theorem \ref{grd_mono001}, we establish the following important lemma.

\begin{lem}
\label{lea.1}
Let $U_\kappa$ be the odd zero-up ground state to \eqref{3.ac}. Suppose that $0<\kappa_1<\kappa_2<1$, then we have
\begin{equation}
U_{\kappa_1}(x)>U_{\kappa_2}(x),\quad x\in\left(0,\frac\pi2\right].
\end{equation}	
\end{lem}

\begin{proof}
At first, we notice that $U_\kappa(0)=0$, $U'_\kappa$ is monotone increasing for $x\in\left(0,\frac\pi2\right).$  By equation \eqref{a.id1} we have
\begin{equation}
\label{a.acahn-1}
U'_\kappa(x)=\frac{\sqrt{(1-U^2_\kappa(x))^2-(1-N^2_{\kappa})^2}}{\sqrt{2}\kappa},\quad \mbox{for}~x\in\left(0,\frac\pi2\right),
\end{equation}
where $N_\kappa$ is the maximal value of $u_\kappa$ in $[0,\frac\pi2]$, i.e., $N_\kappa=U_\kappa(\frac\pi2)$. By \eqref{a.acahn-1} we have
\begin{equation}
\int_0^x\frac{U'_\kappa(x)}{\sqrt{U_\kappa^4(x)-2U_\kappa^2(x)+2N_\kappa^2-N_\kappa^4}}dx=\frac{x}{\sqrt{2}\kappa},\quad x\in\left(0,\frac\pi2\right),
\end{equation}
which is equivalent to
\begin{equation}
\label{a.acahn-2}
\int_0^{U_\kappa(x)}\frac{\kappa}{\sqrt{s^4-2s^2+2N_\kappa^2-N_\kappa^4}}ds=\frac{x}{\sqrt{2}}, \quad x\in\left(0,\frac\pi2\right).
\end{equation}
If $0<\kappa_1<\kappa_2<1$, we have $1>N_{\kappa_1}>N_{\kappa_2}>0$ by equation \eqref{3.def-N}.
This implies that
\begin{align}1>2N_{\kappa_1}^2-N_{\kappa_1}^4>2N_{\kappa_2}^2-N_{\kappa_2}^4>0.\end{align}
Therefore for any $s\in(0,\min\{N_{\kappa_1},N_{\kappa_2}\})$ we have
\begin{equation}
\label{a.acahn-3}
\frac{\kappa_1}{\sqrt{s^4-2s^2+2N_{\kappa_1}^2-N_{\kappa_1}^4}}<\frac{\kappa_2}{\sqrt{s^4-2s^2+2N_{\kappa_2}^2-N_{\kappa_2}^4}}.
\end{equation}
Together with \eqref{a.acahn-2} we derive that $U_{\kappa_1}(x)>U_{\kappa_2}(x)$ for $x\in\left(0,\frac\pi2\right]$. This proves the lemma.
\end{proof}

\begin{proof}[Proof of Theorem \ref{grd_mono001}.]
We shall prove Theorem \ref{grd_mono001} point by point. For point (a), we notice that $u\equiv 0$ is the only odd zero-up solution to \eqref{3.ac} whenever $\kappa\ge1$. Then it is easy to verify that $\tilde E_\kappa=\frac\pi2$ for $\kappa\ge1$.

Next, we consider the point (b). For any $2\pi$-periodic odd zero-up solution of \eqref{3.ac} which is different by $U_\kappa$, we denote its minimal period by $2\pi/m$ and the solution by $u_m,~m\geq2$.
Consider the function
\begin{align}v(y)=u_m(x),\quad y=mx.\end{align}
Then it is not difficult to verify that
\begin{align}v(x)=U_{m\kappa}(x)\quad\mathrm{and}\quad x\in\left(0,\frac\pi2\right).\end{align}
By Lemma \ref{lea.1}, for $m\geq2 $ we have
\begin{equation}
\label{a.g.com}
U_\kappa(x)< U_{m\kappa}(x) \quad \mbox{for}\quad x\in \left(0,\frac\pi2\right).	
\end{equation}
On the other hand, we notice that
\begin{equation}
\label{a.g.com-1}
\begin{aligned}
E^{(0)}_{\kappa}=~&\int_0^{\frac{\pi}{2}}\left(2\kappa^2 U_\kappa'^2+(1-U_\kappa^2)^2\right)dx\\
=~&2\kappa^2 U_\kappa U_\kappa'\mid_{x=0}^{x=\frac\pi2}+\int_0^{\frac{\pi}{2}}\left(-2\kappa^2 U_\kappa U''_\kappa+(1-U_\kappa^2)^2\right)dx\\
=~&\int_0^{\frac{\pi}{2}}(1-U_\kappa^4)dx.
\end{aligned}	
\end{equation}
Using \eqref{a.g.com} we have
\begin{equation}
\label{a.g.com-2}
E^{(0)}_{\kappa}>E^{(0)}_{m\kappa}.
\end{equation}
By equation \eqref{E_jeps002} we get
\begin{align}E^{(0)}_{m\kappa}=E(u_m)=\int_{\mathbb{T}}\left(\frac12(\kappa u_m)^2+\frac14(u_m^2-1)^2\right)dx.\end{align}
Together with \eqref{a.g.com-2} we obtain that
\begin{align}E^{(0)}_{\kappa}=\tilde E_\kappa,\end{align}
and it proves the point (b).

The point (c) follows easily by point (b), Lemma \ref{lea.1} and equation \eqref{a.g.com-1}.

In the end,  we shall show the asymptotics as $\ka \to 0$.
By Proposition \ref{sdy1}, the main part of $U_{\ka}$ on $[-\frac \pi 2, \frac \pi 2]$ is given by
$\tanh(\frac x {\sqrt 2 \ka} )$. The result \eqref{E_eplimit00} then follows from
a simple computation
\begin{align}
\gamma_*={\sqrt 2}  \int_{\mathbb R} (\tanh^2 y -1)^2 dy
 = {\sqrt 2} \int_{\mathbb R} (1-\tanh^2 y) \, d \tanh y
=\frac 43 {\sqrt 2},
\end{align}
where $y = \frac x {\sqrt 2 \ka}$.
\end{proof}


\begin{cor}\label{cor}
For any $0<\ka<1$, if $u_0$ is odd, $2\pi$-periodic (monotonicity of $u_0$ is not required), and $E(u_0)< E_{2\ka}^{(0)}$, then the steady state of \eqref{1.ac} is $\pm U_\ka$, where $U_\ka$ be the odd zero-up ground state to \eqref{3.ac}.
\end{cor}
\begin{proof}
Obviously, we have $E_\ka^{(0)}\leq E(u_0)< E_{2\ka}^{(0)}$.
From the energy dissipation property of \eqref{1.ac}, we can claim that $2\pi$ is the minimal period of the steady state for the initial condition $u_0$.
\end{proof}


\section{Convergence to the steady state}
In this section, we investigate the convergence rate of the solution and characterize the detailed
profiles as $t\to \infty$.

\subsection{Case of $0<\ka<1$}
%
%
We start this subsection with the following result on the spectrum analysis. This is crucial in showing the convergence rate is exponential

\begin{lem}
\label{le4.spectrum}
Let $0<\ka<1$.	Assume $U_\ka$ is the odd zero-up ground state. Then for any $2\pi$-periodic odd function $\phi\in H^1(\mathbb T)$ we have
\begin{equation}
\label{4.s-1}
\int_{\mathbb T}\ka^2|\phi'|^2dx+\int_{\mathbb T}\left(3U_\ka^2-1\right)|\phi|^2dx\geq C\|\phi\|^2_{H^1(\mathbb T)}
\end{equation}
for some universal constant $C>0$.
\end{lem}

\begin{proof}
First of all, we notice that $C\ge0$ due to the fact $U_\ka$ is the odd zero-up ground state. Next, we shall prove that $C>0$ by contradiction. Suppose that $C=0$ then we can find a sequence of odd functions  $\phi_n$ such that
$\|\phi_n\|_{H^1(\mathbb T)}=1$ and
\begin{equation}
\label{4.s-1}
\int_{\mathbb T}\ka^2|\phi_n'|^2dx+\int_{\mathbb T}\left(3U_\ka^2-1\right)|\phi_n|^2dx\leq \frac1n.
\end{equation}
Passing to a subsequence if necessary, we obtain there exists a nontrivial odd function $\phi_*\in H^1(\mathbb T)$ such that $\phi_n$ weakly converges to $\phi_*$ in $H^1(\mathbb T)$ and
\begin{equation}
\label{4.s-2}
\ka^2\phi_*''+(1-3U_\ka^2)\phi_*=0\quad \mathrm{on}\quad \mathbb{T}.
\end{equation}
After direct computations we see that
\begin{equation}
\label{4.s-3}
\begin{aligned}
E_\ka(U_\ka+c\phi_*)=~&E_\ka(U_\ka)+\frac{c^2}{2}\int_{\mathbb T}\left(\ka^2|\phi_*'|^2+(3U_\ka^2-1)|\phi_*|^2\right)dx\\
&+\int_{\mathbb T}\left(c^3U_\ka\phi_*^3+\frac{c^4}{4}\phi_*^4\right)dx
\end{aligned}
\end{equation}
for any real number $c$. Here we have used $U_\ka$ is the odd zero-up ground state. Using \eqref{4.s-2} we see that the second term on the right hand side of \eqref{4.s-3} vanishes, then together with
$E_\ka(U_\ka+c\phi_*)\geq E_k(U_\ka)$ for any $c$, we see that
\begin{equation}
\int_{\mathbb T}U_k\phi_*^3=0.
\end{equation}
It implies that $\phi_*$ must possess a zero point in $(0,\pi),$ denoted by $x_*$. By the well-known Strum Comparison Theorem (see \cite[Theorem VI-1-1]{hs2012} for instance) we derive that any solution of the following equation must have a zero point in $(0,x_*),$
\begin{equation}
\label{4.s-e}
\phi''+(1-U_\ka^2)\phi=0.
\end{equation}
However,  we notice that $U_k$ is a solution of \eqref{4.s-e} and positive in $(0,\pi)$.
Hence we arrive at a contradiction and the lemma is proved.
\end{proof}

With above lemma, we are now able to establish the proof of Theorem \ref{thm1.4}.

%

\begin{proof}[Proof of Theorem \ref{thm1.4}.]
By smoothing estimates we may assume with no loss that $u_0\in C^\infty.$ It is not difficult to check that $u(x,t)$ is a $2\pi$-periodic odd function and also odd symmetric with respect to $x=\pi.$ Therefore
\begin{equation}
\label{4.boundary}
u(0,t)=u(\pi,t)\equiv 0,\quad \forall~ t\geq0.
\end{equation}
Together with that $u_0(x)$ is non-negative in $[0,\pi]$, we conclude that $u(x,t)\geq 0$ for $x\in[0,\pi]$ by Maximum Principle, see \cite[Section 2]{friedman2008} for instance. Similarly, we have $u(x,t)\leq 0$ for $x\in[-\pi,0].$ Now by using the energy conservation we have
\begin{align}
\label{4.monotone}
	\frac d {dt} \left( \frac 12 \| \partial_x u \|_2^2 +\int_{\mathbb T} F(u) dx\right) = - \| \partial_t u \|_2^2.
\end{align}
where $F(u)=\frac14(u^2-1)^2$. It follows that $\| \partial_t u \|_{L_t^1L_x^2} <\infty$ and one can extract a subsequence such that $\partial_t u (t_n ) \to 0$ in $L^2$. By using higher uniform Sobolev estimates one can obtain convergence in higher norms. In particular we can obtain  $u(t_n) \to u_{\infty}$ for some steady state of \eqref{3.ac}. In addition, $u_\infty$ is a $2\pi$-periodic odd function and non-negative for $x\in[0,\pi]$. By the proof of Theorem \ref{thm1.1} we see that $0$ and $U_\kappa$ are the only steady states which are non-negative in $[0,\pi]$. As a consequence, we derive that $u_\infty$ could be either $U_\ka$ or the trivial solution $0$.
	
	
If $E_\ka(u_0)\leq\frac\pi2$ and $u_0(x)\neq0$, using \eqref{4.monotone} we see that
\begin{align}E_\ka(u_\infty)\leq E_\ka(u_0)\leq\frac\pi2.\end{align}
The equality sign holds only $u_\infty=u_0$. While it is known that $E_\ka(0)=\frac\pi2$ and $u_0\neq0$. Then we get $u_\infty=U_\ka$. To obtain exponential convergence, we can take $t_n$ sufficiently large such that $u(t_n)$ is sufficiently close to the steady state $u_{\ka}$. Combined with Lemma \ref{le4.spectrum} we then obtain the exponential convergence. Thus, we finish the whole proof.
\end{proof}

\begin{rem}
If the initial data $u_0$ is $2\pi$-periodic and satisfies
\begin{equation}
\label{4.1}
\left\{\begin{aligned}
&u_0(x)=-u_0(-x),\quad  \forall~ x\in\mathbb{R},\\
&u_0'(x)>0,~u_0(x)=u_0(\pi-x),\quad\quad~\forall~ x\in\left(0,\frac\pi2\right).
\end{aligned}\right.
\end{equation}
Then by Theorem \ref{thm1.4} we can prove that $u(x,t)\to U_\ka$ as $t\to\infty$ whenever $E_\ka(u_0)\leq\frac{\pi}{2}.$ A typical example of the initial data satisfying \eqref{4.1} is $u_0(x)=\sin x$. More examples can be easily constructed along these lines.
\end{rem}

Before we end the study for the case $\ka\in(0,1)$, we present the following result establishing an useful property of the odd zero-up ground state. This part is of independent interest.

\begin{lem}\label{lem-stablity}
	Fix $\ka\in[\frac12,1)$ and assume $U_{\ka}$ is the unique odd zero-up ground state for the equation
	$$\ka^2u''+u-u^3=0\quad \mbox{on}\quad \mathbb T=[-\pi,\pi].$$
	Suppose that $u\in H^1(\mathbb T)$ is an odd function on $\mathbb{T}$ with
	\begin{align}
		E_\ka(u)=\int_{\mathbb T}\left(\frac12\kappa^2(u')^2+\frac14(u^2-1)^2\right)dx<\frac\pi2-C_\ka,
	\end{align}	
    where $C_\ka$ is a positive constant depending on $\ka.$ Then we have
	\begin{equation}
		\label{4.com}
		\min\{\|u-U_{\ka}\|_{H^1(\mathbb{T})},~
		\|u+U_{\ka}\|_{H^1(\mathbb{T})}\}
		\leq C\sqrt{E(u)-E(U_{\ka})},
	\end{equation}	
	where $C>0$ is an absolute constant.
\end{lem}

\begin{rem}
In the special case $\ka=0.9$, we can take $C_\ka=0.001$. For the case $\ka\in(0,\frac12)$, there are multiple steady-states and we shall address this issue elsewhere.
\end{rem}

\begin{proof}
	We first claim that for odd $u\in H^1(\mathbb T)$, when $E(u)-E(U_{\ka})\rightarrow 0$, we must have
	\begin{align}
		\min\{\|u-U_{\ka}\|_{H^1(\mathbb{T})},~
		\|u+U_{\ka}\|_{H^1(\mathbb{T})}\}\rightarrow 0 .
	\end{align}
	We shall prove this by contradiction.
	Suppose the statement is not true, then for some $c_0>0$, there exists a sequence of odd functions $\{u_n\}$ such that
	\begin{equation}
		\label{4.com1}
		E(u_n)-E(U_{\ka})\leq\frac1n,
	\end{equation}
	and
	\begin{equation}
		\label{4.com2}
		\min\{\|u_n-U_{\ka}\|_{H^1(\mathbb{T})},~
		\|u_n+U_{\ka}\|_{H^1(\mathbb{T})}\}
		\geq c_0 > 0.
	\end{equation}
	Using \eqref{4.com1} we can find a universal constant $C$ such that
	\begin{equation}
		\label{4.com3}
		\int_{\mathbb{T}}|u'_n|^2dx+\int_{\mathbb{T}}(1-u_n^2)^2dx\leq C,
	\end{equation}
	which implies that $\{u_n\}$ is a sequence of  odd functions, and bounded in $H^1(\mathbb{T})$. Then we could select a subsequence, still denoted by $\{u_n\}$, such that
	\begin{equation}
		\label{4.com4}
		u_n~\mbox{ converge weakly to}~u_*~\mbox{in}~H^1(\mathbb{T}),
	\end{equation}
	for some odd function $u_*\in H^1(\mathbb{T}).$ By the Rellich Lemma and lower semi-continuity of weak convergence, we have
	\begin{equation}
		E(U_{\ka})\leq E(u_*)\leq \liminf_{n\to+\infty}E(u_n) = E(U_{\ka}).
	\end{equation}
	This implies that $E(U_{\ka})=E(u_*)$. Then we conclude that
	$u_*$ is either $U_{\ka}$ or $-U_{\ka}.$
	Thus
	\begin{equation}
		\label{4.com5}
		u_n~\mbox{strongly converge to}~U_{\ka}~\mbox{or}~-U_{\ka}~\mbox{in}~H^1(\mathbb T).
	\end{equation}
	This contradicts to \eqref{4.com2}. Therefore,  the claim holds.
	
	By using the claim, to establish \eqref{4.com}, it suffices for us to consider the situation
	\begin{align}
		\min\{\|u_n-U_{\ka}\|_{H^1(\mathbb{T})},~
		\|u_n+U_{\ka}\|_{H^1(\mathbb{T})}\} \ll 1.
	\end{align}
	In this case, without loss of generality we assume that $u'(0)\ge 0$ and denote $\eta = u-U_{\ka}$. Then it is not difficult to check that
	\begin{align}
		E(u) - E(U_{\ka}) = \frac 1 2 \int_{\mathbb T} \left[\ka^2(\partial_x\eta)^2+(3U_{\ka}^2-1)\eta^2\right] dx + O(\|\eta\|_{H^1}^3)
	\end{align}
	By Lemma \ref{le4.spectrum}, we have
	\begin{align}
		\frac 1 2 \int_{\mathbb T}\left[\ka^2(\partial_x\eta)^2+(3U_{\ka}^2-1)\eta^2\right] dx \geq C \|\eta\|_{H^1}^2,
	\end{align}
	where $C>0$ is an absolute constant.
	The desired conclusion follows easily.
\end{proof}

\subsection{Case of $\ka\ge 1$}
In the case of $\ka\ge 1$, we consider the a general Allen-Cahn equation
\begin{align} \label{5sdy1}
	\begin{cases}
		\partial_t u = -\ka^2 \Lambda^{\gamma} u -(u^3-u), \qquad (x,t) \in \mathbb T \times (0,\infty),\\
		u\Bigr|_{t=0} =u_0,
	\end{cases}
\end{align}
where $\Lambda^{\gamma}= (-\partial_{xx})^{\gamma/2}$ is the fractional Laplacian of order $\gamma\in (0,2)$. When $\gamma=2$ it coincides with $-\partial_{xx}$.

\begin{prop}[Preliminary properties of steady states for $0<\gamma<1$] \label{5sdy2}
	Let $0<\gamma<2$ and $\ka >0$. Suppose $\phi:\mathbb T \to \mathbb R$ is $C^{1,1}$ and satisfies
	\begin{align}
		- \ka^2 \Lambda^{\gamma} \phi - (\phi^2-1) \phi=0.
	\end{align}
	Then $\phi \in C^{\infty}(\mathbb T)$, and only one of the following occur:
	\begin{itemize}
		\item $\phi \equiv 1$;
		\item $\phi \equiv -1$ ;
		\item $\|\phi\|_{\infty} <1$.
	\end{itemize}
\end{prop}
\begin{proof}
	This follows from the usual maximum principle argument using the expression
	\begin{align}
		(\Lambda^{\gamma} \phi)(x) = C_{\gamma} \sum_{n \in \mathbb Z} \operatorname{PV} \int_{|y|<\pi}
		\frac {\phi(x) - \phi(y)} {|x-y+2n\pi|^{1+\gamma} } dy.
	\end{align}
	Alternatively one can also derive the result using harmonic extension.
\end{proof}

To state the next result, we introduce the Fourier projection operators $\Pi_1$,  $\Pi_{\ge 2}$ such that
for $f = \sum_{m\ge 1} f_m \sin m x$ (assume the series converges sufficiently fast),
\begin{align} \label{Pidef}
	&\Pi_{1} f = f_1 \sin x; \qquad
	\Pi_{\ge 2} f = \sum_{m\ge 2} f_m \sin mx.
\end{align}
In other words $\Pi_1$ is the projection to the first sine-mode, and
$\Pi_{\ge 2}$ simply removes the first Fourier mode in the sine series expansion.

\begin{thm} \label{5sdy3}
	Let $\ka\ge 1$ and $0<\gamma\le 2$.
	Assume $u_0$ is $2\pi$ periodic, odd and bounded.
	Suppose $u$ is the solution to \eqref{5sdy1} corresponding to the initial data $u_0$.
	If $\ka >1$, we have exponential decay
	\begin{align}
		&\|u(t,\cdot )\|_2 \le \|u_0\|_2 e^{-(\ka^2-1) t}, \qquad \forall~t\ge 0;  \\
		&\|u(t,\cdot) \|_{H^{10}} \le \beta_1 e^{-(\ka^2-1) t}, \qquad\forall~t\ge \frac 12; \label{5sdy3_a}\\
		&\|\Pi_{\ge 2} u(t,\cdot) \|_{H^{10}} \le \beta_2 e^{-\eta_1 t}, \qquad\forall~t\ge \frac 12, \label{5sdy3_b}
	\end{align}
	where $\beta_1>0$, $\beta_2>0$ depend on ($u_0$, $\gamma$, $\ka$), and
	$\Pi_{\ge 2}$ was defined in \eqref{Pidef}. The constant $\eta_1> \ka^2-1$ is given by
	\begin{align}
		\eta_1 = \min \{ \ka^2 2^{\gamma}-1, \; 3(\ka^2-1) \}.
	\end{align}
	For $\ka=1$, we have algebraic decay:
	\begin{align}
		&\| u (t,\cdot)\|_2 \le  \frac{\sqrt{\pi} \|u_0\|_2}{\sqrt{t\|u_0\|_2^2+\pi}},\qquad\forall~t\ge 0;  \label{5sdy3_c} \\
		& \|u(t,\cdot) \|_{H^{10}} \le  \beta_3 t^{-\frac 12},\qquad\qquad\forall~t\ge \frac 12; \label{5sdy3_e}\\
		& \|\Pi_{\ge 2} u(t,\cdot) \|_{H^{10}} \le  \beta_4 t^{-\frac 32}, \qquad\forall~t\ge \frac 12, \label{5sdy3_d}
	\end{align}
	where $\beta_3>0$, $\beta_4>0$ depend on ($u_0$, $\gamma$).
\end{thm}
\begin{rem}
	For $\ka>1$, higher (i.e. $H^m$, $m>10$) Sobolev norms of $u$ also decay exponentially but we shall not dwell on this issue here.
	Note that we state the decay result for $t\ge \frac 12$ to allow the smoothing effect to kick in.
	The number $\frac 12$ is for convenience only and it can be replaced by any other $t_0>0$ with
	suitable adjustment of the corresponding pre-factors in the estimates.
\end{rem}
\begin{proof}
	First we note that for bounded initial data, local and global wellposedness is not an issue and we focus
	solely on the decay estimates.
	
	For the $L^2$ decay estimates, first we  assume $u_0$ is smooth, and in particular has a finite sine-series expansion.
	It follows that $u(t)$ must have a spectral gap. By using the Poincar\'e inequality we have
	\begin{align}
		\| \Lambda^{\frac {\gamma} 2} u \|_{L^2(\mathbb T)}  \ge \| u \|_{L^2(\mathbb T)}.
	\end{align}
	By using the above estimate and the fact that $\ka\ge 1$, we  obtain
	\begin{equation}
	\begin{aligned}
		\frac 12 \frac d {dt} ( \| u \|_2^2)
		& =- \ka^2 \| \Lambda^{\frac{\gamma}2} u \|_2^2 + \| u\|_2^2 - \| u\|_4^4  \le - (\ka^2-1)\| u\|_2^2 -\|u \|_4^4 \\
		&  \le - (\ka^2-1)\| u\|_2^2  - \frac 1 {2\pi} \|u\|_2^4,
	\end{aligned}
    \end{equation}
	where in the last step we have used the H\"older's inequality.
	Then, we derive that in the case of $\ka>1$,
	\begin{equation}
		\| u \|_2 \leq \| u_0 \|_2 \,e^{-(\ka^2-1)t},
	\end{equation}
	while in the case of $\ka = 1$,
	\begin{equation}
		\| u \|_2 \leq \frac{\sqrt{\pi} \|u_0\|_2}{\sqrt{t\|u_0\|_2^2+\pi}}.
	\end{equation}
	By a simple approximation argument, both estimates also hold under the assumption that
	$u_0 \in L^{\infty}$.
	
	We now show \eqref{5sdy3_a}. First by smoothing estimates and interpolation, we have
	\begin{align}
		\| \partial_x^{10} (u(t,\cdot) ) \|_{2} \le \alpha_1 e^{-\ka_1 t }, \qquad\forall\, t\ge \frac 12,
	\end{align}
	where $\alpha_1>0$ depends on ($u_0$, $\gamma$, $\ka$), and $\ka_1>0$ depends only
	on $\ka$.  It follows easily that
	\begin{align}
		\| \partial_x^{10} (u^3(t,\cdot) ) \|_{2} \le \alpha_2 e^{-\ka_1 t }
		\| \partial_x^{10} u(t,\cdot) \|_{2}, \qquad\forall\, t\ge \frac 12,
	\end{align}
	where $\alpha_2>0$ depends on ($u_0$, $\gamma$, $\ka$).  We now compute
	for $t\ge \frac 12$,
	\begin{align}
		\frac 12 \frac d {dt} ( \| \partial_x^{10} u(t,\cdot) \|_2^2 )
		& \le -\ka^2 \| \Lambda^{\frac {\gamma}2} \partial_x^{10} u \|_2^2
		+ \| \partial_x^{10} u \|_2^2 +
		\|\partial_x^{10}  (u^3) \|_2 \| \partial_x^{10} u\|_2 \notag \\
		& \le ( - (\ka^2-1) + \alpha_2 e^{-\ka_1 t} ) \| \partial_x^{10} u  \|_2^2.
	\end{align}
	Integrating in time then yields \eqref{5sdy3_a}.
	
	The proof of \eqref{5sdy3_b} is similar.  Note that for all $t\ge \frac 12$,
	\begin{align}
		\|\partial_x^{10} \Pi_{\ge 2} (u^3 (t,\cdot ) ) \|_2
		\le
		\|
		\partial_x^{10} (u^3 (t,\cdot ) ) \|_2
		\le \alpha_3 e^{-3(\ka^2-1) t},
	\end{align}
	where $\alpha_3>0$ depends on  ($u_0$, $\gamma$, $\ka$).
	With this we compute:
	\begin{equation}
	\begin{aligned}
		&\frac 12 \frac d {dt} ( \| \partial_x^{10} \Pi_{\ge 2} u(t,\cdot) \|_2^2 )\\
		& \le -\ka^2 \| \Lambda^{\frac {\gamma}2} \partial_x^{10} \Pi_{\ge 2} u \|_2^2
		+ \| \partial_x^{10} \Pi_{\ge 2} u \|_2^2 +
		\|\partial_x^{10}  (u^3) \|_2 \| \partial_x^{10} \Pi_{\ge 2}  u\|_2 \\
		& \le ( - (\ka^22^{\gamma}-1)  ) \| \partial_x^{10} \Pi_{\ge 2} u  \|_2^2
		+ \alpha_3 e^{-3(\ka^2-1) t} \| \partial_x^{10} \Pi_{\ge 2}  u\|_2.
	\end{aligned}
    \end{equation}
	Thus \eqref{5sdy3_b} follows from a simple ODE argument.
	
	Finally \eqref{5sdy3_d} follows from working with the system
	\begin{align}
		\partial_t \Pi_{\ge 2} u = -\ka^2 \Lambda^\gamma \Pi_{\ge 2} u +
		\Pi_{\ge 2} u -\Pi_{\ge 2} (u^3),
	\end{align}
	and bootstrapping estimates using \eqref{5sdy3_c}.  The estimate \eqref{5sdy3_e} is
	obvious. We omit  the details.
\end{proof}

We turn now to some (by now) standard log-convexity results.

\begin{prop} (Log convexity for an almost-linear model) \label{5sdpro1}
	Suppose $\mathbb H$ is a real Hilbert space with inner product $\langle\cdot, \cdot\rangle$ and norm
	$\|\cdot \|$. Let $A$ be a symmetric operator on $\mathbb H$ with domain $\mathcal D(A)$.
	Let $T>0$ and $u \in C_t^1([0, T], \, \mathbb H)$ satisfy $u(t) \in \mathcal D(A)$
	for each $0<t<T$, and
	\begin{align}
		\| \partial_t u + A u \| \le \alpha(t) \| u (t) \|, \qquad\forall\, 0<t<T,
	\end{align}
	where $\alpha(t) \ge 0$ satisfies
	\begin{align}
		\int_0^T \alpha(t)^2 dt <\infty.
	\end{align}
	Denote $m(t) = \| u(t) \|^2$.
	Then $m$ is log-convex:
	\begin{align}
		m(t) \le  e^{\int_0^T (4\alpha(s)+s\alpha(s)^2) ds} m(0)^{1-\frac t T} m(T)^{\frac t T},
		\qquad\forall\, 0\le t \le T.
	\end{align}
	It follows that either $m (t)\equiv 0$ on $[0, T]$ or  $m(t) >0$ for all $t\in [0,T]$.
\end{prop}
\begin{proof}
	First we assume that $m(t)>0$ for all $0\le t\le T$.
	Denote $f=\partial_t u+Au$
	so that
	\begin{align}
		\partial_t u = -Au +f.
	\end{align}
	Denote
	\begin{align}
		b_1 = \frac 2m \langle u,f\rangle.
	\end{align}
	Then clearly
	\begin{align}
		&\frac d {dt} \left( \ln m +\int_t^T b_1(s) ds \right) = \frac 2 m \langle u,-Au\rangle;
	\end{align}
	and
	\begin{equation}
	\begin{aligned}
		&\frac 14  \frac {d^2}{dt^2} \left(\ln m +\int_t^T b_1(s) ds\right)\\
		&= \frac {\langle u_t, -Au\rangle}m - \frac {\langle u, -Au\rangle\langle u, -Au+f\rangle } {m^2}\\
		&= \frac {m \|Au\|^2- |\langle u,-Au\rangle|^2} {m^2}
		+ \frac{\langle f,-Au\rangle }m + \frac {\langle u,Au\rangle \langle u,f\rangle} {m^2}.
	\end{aligned}
    \end{equation}
	We decompose $Au=c_1 u +c u^{\perp}$ where $c_1, c \in \mathbb R$ and $u^{\perp}$ is
	a unit vector orthogonal to $u$.  Plugging this into the last expression, we obtain
	\begin{equation}
	\begin{aligned}
		\frac 14  \frac {d^2}{dt^2}\left(\ln m +\int_t^T b_1(s) ds\right )
		&= \frac {c^2} m - \frac {c \langle f, u^{\perp} \rangle} m  \\
		&\ge \frac { (|c|- \frac 12 |\langle f, u^{\perp}\rangle|)^2} {m}
		-\frac 14 \frac{ |\langle f, u^{\perp}\rangle|^2} m   \\
		& \ge -\frac 14 \alpha(t)^2.
	\end{aligned}
    \end{equation}
	It follows that $r(t)$ is convex, where
	\begin{align}
		r(t) = \ln m(t) + \underbrace{\int_t^T b_1(s) ds + \int_t^T \int_{\tau}^T \alpha^2(s) ds d\tau}_{=:b(t)}.
	\end{align}
	From the convexity of $r(t)$ we deduce
	\begin{align}
		r(t) \le (1-\frac t T) r(0) + \frac t T r(T).
	\end{align}
	This implies that
	\begin{align}
		\ln m(t) \le (1-\frac t T) \ln m(0) + \frac t T \ln m(T)
		+(1-\frac t T) b(0) +\frac t T b(T) - b(t).
	\end{align}
	Since $b(T)=0$ and $\int_t^T \int_s^T \alpha(\tau)^2 d\tau ds \ge 0$, we obtain
	\begin{equation}
		\begin{aligned}
		(1-\frac t T) b(0) +\frac t T b(T) - b(t)
		& = (1-\frac t T)b(0) -b(t) \\
		& \le 2 \int_0^T |b_1(s)|ds  + \int_0^T \int_{\tau}^T \alpha^2 (s) ds d\tau \\
		& \le \int_0^T (4 \alpha(s) +s\alpha(s)^2)ds.
	\end{aligned}
    \end{equation}
	Thus the desired inequality holds under the assumption that $m(t)>0$ for all $t \in [0,T]$.
	
	Now we show how to remove this assumption. Assume that $m(t)$ is not identically zero.
	Since $m(t)\ge 0$ is a continuous function of $t$
	and $m(t)$ is not identically zero,
	we may assume that there exists $t_0\in [0, T]$ such that $m(t_0)>0$. By a continuity argument we can assume $t_0 \in (0,T)$. Now denote
	\begin{align}
		&t_+= \sup\{t:\; t>t_0 \text{ such that $m(s)>0$ for all $t_0\le s \le t$} \}; \\
		&t_-= \inf\{t: \; t<t_0 \text{ such that $m(s)>0$ for all $t\le s \le t_0$} \}.
	\end{align}
	If $t_+<T$, then we have $m(t_+)=0$ with $m(t)>0$ for all $t_0\le t <t_+$. By using a version of
	the proved
	inequality on the interval $[t_0, t_+-\eta]$ (note that $m(t)>0$ for all $t_0\le t\le t_+-\eta$ and
	thus we can use the proved inequality with the interval $[0,T]$ now
	replaced by $[t_0, t_+-\eta]$) and sending $\eta \to 0$, we clearly obtain a contradiction.
	If $t_+=T$ and $m(T)=0$, we also obtain a contradiction by a similar argument.  By a similar reasoning we obtain $t_-=0$ and $m(0)>0$. Thus we have proved that $m(t)>0$ for all
	$t\in [0,T]$.
\end{proof}

\begin{lem} \label{5sdy3_le1}
	For any $0<\gamma\le 2$, there exits $\eta_0=\eta_0(\gamma)>0$ such that
	the following hold for any smooth $2\pi$-periodic \emph{odd} function $u$ on $\mathbb T$:
	\begin{align}
		\int_{\mathbb T} u^3 (-\partial_{xx} )^{\frac {\gamma}2} u  dx
		\ge \eta_0 \|u\|_4^4.
	\end{align}
	For $\gamma=2$, we can take $\eta_0=3/4$.
\end{lem}
\begin{proof}
	This follows from a general result proved in \cite{Li2013}.  For $\gamma=2$ we give
	a direct proof as follows (below we write $\int_{\mathbb T} dx$ as $\int$)
	\begin{align}
		\int u^3 (-\partial_{xx} u)
		= 3 \int (\partial_x u)^2 u^2 =\frac 3 4 \int \Bigl(\partial_x ( |u| u) \Bigr)^2 dx
		\ge \frac 34 \| |u| u \|_2^2 =\frac 34 \|u\|_4^4.
	\end{align}
	Note that in the above we took advantage of the odd symmetry since the function
	$v= |u| u$ is still odd on $[-\pi, \pi]$.  Note that regularity is not an issue here since
	the function $g(z)=|z|z$ is nice.
\end{proof}

\begin{thm}[Log convexity of $L^2$ mass for the nonlinear case] \label{5sdy3a}
	Let $\ka>0$ and $0<\gamma\le 2$.
	Assume $u_0$ is $2\pi$ periodic, odd and bounded. To avoid triviality assume $\|u_0\|_2>0$
	so that $u_0$ is not identically zero.
	Suppose $u$ is the solution to \eqref{5sdy1} corresponding to the initial data $u_0$.
	Denote $m(t) = \|u (t)\|_{L_x^2(\mathbb T)}^2$.
	Then the following hold.
	\begin{itemize}
		\item If $0<\ka <1$, then $m(t)$ is log-convex on any interval $0\le t_1 <t_2$:
		\begin{align}
			m(t) \le e^{c_1 \cdot (t_2-t_1)^2} m(t_1)^{1-\frac{t-t_1}{t_2-t_1}}
			m(t_2)^{\frac {t-t_1}{t_2-t_1} }, \qquad\forall\, t\in (t_1,t_2),
		\end{align}
		where $c_1>0$ is a constant depending only on ($\|u_0\|_{\infty}$, $\gamma$,
		$\ka$).
		\item If $\ka >1$, then $m(t)$ is log-convex on any interval $0\le t_1 <t_2$:
		\begin{align}
			m(t) \le  c_2 m(t_1)^{1-\frac{t-t_1}{t_2-t_1}}
			m(t_2)^{\frac {t-t_1}{t_2-t_1} }, \qquad\forall\, t\in (t_1,t_2),
		\end{align}
		where $c_2>0$ is a constant depending only on ($\|u_0\|_{\infty}$, $\gamma$,
		$\ka$).
		
		\item If $\ka =1$, then $m(t)$ is log-convex on any interval $0\le t_1 <t_2$:
		\begin{align}
			m(t) \le  (1+t_2-t_1)^{c_3}m(t_1)^{1-\frac{t-t_1}{t_2-t_1}}
			m(t_2)^{\frac {t-t_1}{t_2-t_1} }, \qquad\forall\, t\in (t_1,t_2),
		\end{align}
		where $c_3>0$ is a constant depending only on ($\|u_0\|_{\infty}$, $\gamma$,
		$\ka$).

		\item For each $0<\gamma\le 2$, there is $\ka_0=\ka_0(\gamma)>0$ such that
		if $\ka\ge \ka_0$, then we have sharp log-convexity, i.e.:
		on any interval $0\le t_1 <t_2$:
		\begin{align} \label{5sdy3a_e3}
			m(t) \le   m(t_1)^{1-\frac{t-t_1}{t_2-t_1}}
			m(t_2)^{\frac {t-t_1}{t_2-t_1} }, \qquad\forall\, t\in (t_1,t_2),
		\end{align}
		Furthermore for $\gamma=2$, we can choose $\ka_0(2) = 2/\sqrt 3$.
		
	\end{itemize}
\end{thm}
\begin{proof}
	First we consider the case $0<\ka <1$. Observe that
	\begin{align}
		\| \partial_t u + \ka^2 \Lambda^{\gamma} u-u
		\|_2 \le \|u (t) \|_{L_x^{\infty} }^2 \|u(t) \|_2.
	\end{align}
	It is not difficult to check that
	\begin{align}
		\sup_{0\le t <\infty} \|u(t) \|_{L_x^{\infty}} \le \tilde c_1,
	\end{align}
	where $\tilde c_1>0$ depends only on ($\|u_0\|_{\infty}$, $\gamma$, $\ka$).
	Thus the result follows from Proposition \ref{5sdpro1}.
	
	Now for $\ka> 1$, we observe that by using Theorem \ref{5sdy3} (note that for $0<s\le \frac 12$ we have uniform control of $L^{\infty}$-norm), it holds that
	\begin{align}
		\|u(s) \|_{L_x^{\infty} } \le \tilde c_2 e^{-\theta_1 s}, \qquad \forall\, s\ge 0,
	\end{align}
	where $\tilde c_2>0$ depends only on ($\|u_0\|_{\infty}$, $\gamma$, $\ka$) and
	$\theta_1$ depends only on ($\ka$, $\gamma$). Thus
	\begin{align}
		\tilde \alpha(s) : = \|u(s) \|_{L_x^{\infty}}^2 \le \tilde c_2^2 e^{-2\theta_1 s}, \qquad\forall\,
		s\ge 0.
	\end{align}
	On any time interval $[t_1, t_2]$ with $0\le t_1 <t_2$, in order to apply Proposition
	\ref{5sdpro1},  we note for $t\ge 0$,
	\begin{align}
		\alpha(t)=\tilde \alpha(t_1+t) \le \tilde c_2^2 e^{-2\theta_1 t}.
	\end{align}
	Thus the desired result follows for $\ka >1$.
	
	The case for $\ka=1$ follows similarly from Theorem \ref{5sdy3} and Proposition
	\ref{5sdpro1}. The main observation is that $\|u(s) \|_{L_x^{\infty}} =O( (1+s)^{-\frac 12})$
	for $s\ge 0$.
	
	Finally we turn to the proof of \eqref{5sdy3a_e3}. We shall appeal to a more ``nonlinear" proof
	as follows. Denote  $m(t) = \| u(t) \|_{L_x^2}^2$ and $A= \ka^2 (-\partial_{xx})^{\frac{\gamma}2} - 1$.  Thus we have
	\begin{align}
		\partial_t u = -A u -u^3.
	\end{align}
	Denote $m^{\prime}=\frac d {dt} m$ and $m^{\prime\prime} = \frac {d^2}{dt^2} m$.
	It is not difficult to check that (below $\langle\cdot, \cdot\rangle$ denotes the usual $L^2$ inner
	product, and $u_t=\partial_t u$)
	\begin{align}
		& m^{\prime}= 2\langle u, u_t\rangle= -2 \langle u, Au\rangle  -2\|u\|_4^4;\\
		&m^{\prime\prime} =-4\langle u_t, Au\rangle-8\langle u^3,u_t\rangle.
	\end{align}
	Thus
	\begin{align}
		& \frac 12 m^{\prime} = \langle u, u_t\rangle; \\
		& \frac 14 m^{\prime\prime}= \|u_t\|_2^2 -\langle u^3, u_t\rangle; \\
		& \frac 14 m^{\prime\prime} m - \frac 1 4 (m^{\prime})^2
		= \|u_t\|_2^2 \|u\|_2^2 - |\langle u, u_t\rangle|^2 +(\langle u^3,Au\rangle +\|u\|_6^6)\|u\|_2^2.
	\end{align}
	It remains for us to verify
	\begin{align}
		\langle u^3, Au\rangle = \ka^2 \int_{\mathbb T}u^3 (-\partial_{xx})^{\frac{\gamma}2} u dx
		- \|u\|_4^4 \ge 0.
	\end{align}
	This in turn follows from Lemma \ref{5sdy3_le1}.
\end{proof}

\begin{cor}[No finite time extinction of $L^2$ mass]  
	Let $\ka>0$ and $0<\gamma\le 2$.
	Assume $u_0$ is $2\pi$ periodic, odd and bounded. To avoid triviality assume $\|u_0\|_2>0$
	so that $u_0$ is not identically zero.
	Suppose $u$ is the solution to \eqref{5sdy1} corresponding to the initial data $u_0$.
	Then $\|u(t) \|_{L_x^2}>0$ for any $0\le t<\infty$.
\end{cor}
\begin{proof}
	This follows easily from Proposition \ref{5sdy3a}.
\end{proof}

\begin{thm}[Profiles as $t\to \infty$] \label{5sdy4}
	Let $\ka\ge 1$ and $0<\gamma\le 2$.
	Assume $u_0$ is $2\pi$ periodic, odd and bounded.  To avoid triviality assume $\|u_0\|_2>0$
	so that $u_0$ is not identically zero.
	Suppose $u$ is the solution to \eqref{5sdy1} corresponding to the initial data $u_0$.  Then
	the following hold.
	
	\begin{itemize}
		\item Case $\ka>1$. For all $t\ge 1$, we have
		\begin{align} \label{5sdy4_1}
			u(x,t) = e^{-(\ka^2-1) t} \alpha_* \sin x + r(t),
		\end{align}
		where the constant $\alpha_*$ depends on ($u_0$, $\gamma$, $\ka$).
		The remainder
		term $r(t)$ has the estimate
		\begin{align}
			\| r(t) \|_{H^{10}} \le \tilde \alpha e^{-\eta_1 t}, \qquad\forall\, t\ge 1,
		\end{align}
		with $\tilde \alpha>0$ depends only on ($u_0$, $\gamma$, $\ka$), and
		$\eta_1 =\min\{ \ka^2 2^{\gamma} -1, \; 3(\ka^2-1) \}>\ka^2-1$.
		
		\item Case $\ka=1$. For all $t\ge 1$, we have
		\begin{align} \label{5sdy4_2}
			u(x,t) = t^{-\frac 12} \beta_* \sin x + r_1(t),
		\end{align}
		where the constant $\beta_*$ depends on ($u_0$, $\gamma$). If $\beta_*=0$, then the remainder
		term $r_1(t)$ has the estimate
		\begin{align}
			\| r_1(t) \|_{H^{10}} \le \tilde \beta t^{-1} \sqrt{\ln (t+2)}, \qquad\forall\, t\ge 1,
		\end{align}
		with $\tilde \beta>0$ depends only on ($u_0$, $\gamma$). If $\beta_* \ne 0$, then the remainder
		term $r_1(t)$ has the estimate
		\begin{align}
			\| r_1(t) \|_{H^{10}} \le \tilde \beta t^{-\frac 32} {\ln (t+2)}, \qquad\forall\, t\ge 1,
		\end{align}
		with $\tilde \beta>0$ depends only on ($u_0$, $\gamma$).
	\end{itemize}
\end{thm}
\begin{rem}
Clearly, Theorem \ref{thm1.3} follows from above result immediately. Note that for $\ka>1$ and generic nontrivial odd periodic $u_0$ we could have
$\alpha_*=0$. An easy example is $u_0(x) =\sin 2x$.  On the other hand, a further interesting
question is to investigate whether the following scenario is possible: namely if we denote
\begin{align}
		\alpha_1(t)= \int_{\mathbb T} u(t,x) \sin x dx.
\end{align}
then for some $t=t_c$, $\alpha_1(t)=0$ for $t\ge t_c$, and $\alpha_1(t)\ne 0$ for $t<t_c$ with $t_c-t$ sufficiently small.

\end{rem}

\begin{proof}
	We first consider $\ka>1$. Write
	\begin{align}
		u= \Pi_1 u + \Pi_{\ge 2} u,
	\end{align}
	where the operators $\Pi_1$, $\Pi_{\ge 2}$ were defined in \eqref{Pidef}. By Theorem \ref{5sdy3}
	the term $\Pi_{\ge 2} u$ has the desired decay for $t\ge 1$ and can be included in the
	remainder $r(t)$.  Thus we only need to treat the single-mode part $\Pi_1 u$. Denote
	\begin{align}
		&\Pi_1 u(t) = a(t) \sin x, \qquad a(t) = \frac 1 {\pi} \int_{\mathbb T} \Pi_1 u(t,x) \sin x dx;\\
		&\Pi_1(-u^3(t)) = b(t) \sin x, \qquad b(t) =
		\frac 1 {\pi} \int_{\mathbb T} \Pi_1 ( -u^3(t,x) ) \sin x dx.
	\end{align}
	By Theorem \ref{5sdy3},  we have for some $\tilde C>0$ depending only
	on ($u_0$, $\gamma$, $\ka$),
	\begin{align}
		&|b(t)| \le \tilde C e^{-3(\ka^2-1) t}, \qquad \forall\, t\ge \frac 12.
	\end{align}
	Clearly we have
	\begin{align}
		\frac d {dt} a(t)= -(\ka^2-1) a (t)+ b(t).
	\end{align}
	We then write for $t\ge 1$,
	\begin{align}
		a(t) & = e^{-(\ka^2-1) (t-\frac 12)} a(\frac 12)
		+ \int_{\frac 12}^t e^{-(\ka^2-1) (t-s) } b(s) ds \notag \\
		& = e^{-(\ka^2-1) t}
		\Bigl( e^{\frac12 (\ka^2-1)} a(\frac 12)
		+ \int_{\frac 12}^{\infty} e^{(\ka^2-1) s} b(s) ds \Bigr)+\tilde r(t),
	\end{align}
	where
	\begin{align}
		|\tilde r(t) | \le e^{-(\ka^2-1)t} \int_t^{\infty}
		e^{(\ka^2-1) s} |b(s) | ds =O(e^{-3(\ka^2-1)t}).
	\end{align}
	Clearly then \eqref{5sdy4_1} follows.

	The proof of \eqref{5sdy4_2} is slightly more intricate. We only need to treat the piece
	$\Pi_1 u$ since the part $\Pi_{\ge 2} u$ can be included in the remainder term
	$r_1(t)$.  Observe that for $t\ge \frac 12$, by Theorem \ref{5sdy3} we have
	\begin{align}
		u(t)^3 = (\Pi_1 u(t)+ \Pi_{\ge 2} u(t) )^3 = (\Pi_1 u)^3  +\tilde r(t),
	\end{align}
	where
	\begin{align}
		\|\tilde r(t) \|_{H^{10} }=
		O(t^{-\frac 52}), \qquad\forall\, t\ge \frac 12.
	\end{align}
	Denote $\Pi_1 u(t) = a(t) \sin x$, clearly
	\begin{align}
		\Pi_1 ( (\Pi_1 u(t) )^3) =\frac 34 a(t)^3  \sin x.
	\end{align}
	For $a(t)$ we have the ODE
	\begin{align} \notag
		\frac d {dt} a(t) = -\frac 3 4 a(t)^3 + b(t), \qquad t\ge \frac 12,
	\end{align}
	where $|b(t) | = O(t^{-\frac 52})$.
	
	Denote $\theta(t)=a(t)^2$. We clearly have
	\begin{align}
		\frac d {dt} \theta(t) = -\frac 32 \theta(t)^2 + O(t^{-3}).
	\end{align}
	
	By Proposition \ref{prTp8321} proved below, we have  for $t\ge 3$,
	\begin{align}
		\theta(t) = \frac {\theta_*} {t} + O(t^{-2} \ln t),
	\end{align}
	Note that $\theta_*\ge 0$ since $\theta(t)$ is always nonnegative.
	
	Now if $\theta_*=0$ we can take $\beta_*=0$ and the desired result follows easily.
	If $\theta_*>0$, then $|a(t)| \sim t^{-1/2}$ for $t$ large. By continuity it can
	only take one sign. Thus we obtain
	$\beta_*=\sqrt{\theta_*}$ or $\beta_*=-\sqrt{\theta_*}$.  The estimate for
	the remainder term is trivial. We omit the details.
\end{proof}
\begin{lem} \label{leTp5212}
	Assume $T_0\ge 1$. Suppose $\theta:\, [T_0,\infty)\to [0,\infty)$ is continuously differentiable
	and safisfy
	\begin{align}
		0<\limsup_{t\to\infty} t\theta(t)  <\infty; \qquad  \theta^{\prime}(t) \ge -\frac 32 \theta^2 (t) -t^{-2.2}, \quad\forall\, t>T_0.
	\end{align}
	Then  we have
	\begin{align}
		\liminf_{t\to \infty} t \theta(t) >0.
	\end{align}
\end{lem}
\begin{proof}
	It is natural to appeal to a maximum principle argument.  Denote
	\begin{align} \label{tmpF3321a}
		2\theta_0 = \limsup_{t\to \infty} \theta(t) t >0.
	\end{align}
	Let
	\begin{align}
		\Omega(t) = \theta(t) - \eta_0 t^{-1},
	\end{align}
	where $\eta_0>0$ satisfies
	\begin{align}
		\eta_0\le \min\{ \frac 12 \theta_0, \; \frac 1{10} \}.
	\end{align}
	By \eqref{tmpF3321a}, we can choose $t_0>0$ sufficiently large such that
	\begin{align}
	\theta(t_0) \ge \frac {\theta_0} {t_0}, \qquad
	 \frac {\eta_0}{2t_0^2} -t_0^{-2.2} >0.
	\end{align}
	Note that the second condition above guarantees that
	\begin{align}
		\frac {\eta_0} {2t^2} - t^{-2.2}>0, \qquad\forall\, t\ge t_0.
	\end{align}
	Now consider $\Omega(t)$ on the time interval $[t_0,\infty)$. If $\Omega(t)>0$
	for all $t\ge t_0$ we are done. Otherwise there exists some time $t_1>t_0$ such that
	$\Omega(t_1)=0$. But then clearly
	\begin{align}
	\theta(t_1) = \frac {\eta_0} {t_1}; \qquad
	\Omega^{\prime}(t_1) = -\frac 32 \frac {\eta_0^2} {t_1^2} + \frac {\eta_0} {t_1^2}
	-t_1^{-2.2} \ge \frac {\eta_0}{2t_1^2}-t_1^{-2.2}>0.
	\end{align}
	Thus $\Omega(t)$ continuous to be positive a little bit past $t_1$. This argument then guarantees
	that $\Omega(t)\ge 0$ for all $t\ge t_0$.
\end{proof}

\begin{lem} \label{leTp5214}
	Assume $T_0\ge 1$ and $0<\ka_0<1$. Suppose $\theta:\, [T_0,\infty)\to [0,\infty)$
	satisfies
	\begin{align}
		\theta(t) \le \frac {\ka_0} t
		\quad\mathrm{and}\quad
		 \theta(t) \le \frac 12 \int_t^{\infty} \theta(s)^2 ds + \frac 1 {2t^2}, \qquad \forall\,
		t\ge T_0.
	\end{align}
	Then there exists a constant $C_1>0$ depending on $\ka_0$ and $T_0$, such that
	\begin{align}
		\theta(t) \le \frac {C_1} {t^2}, \qquad\forall\, t\ge T_0.
	\end{align}
\end{lem}
\begin{proof}
	We begin by noting that, if we  assume
	\begin{align}
		\theta(t) \le \frac {\alpha} {t}, \qquad\forall\, t\ge \max\{ T_0, \;\frac 1 {\alpha} \}.
	\end{align}
	then we obtain
	\begin{align}
		\theta(t) \le \frac 12 \cdot \frac {\alpha^2} t+ \frac 1{2t^2} \le
		\frac {\alpha^2} {t}, \qquad\forall\, t\ge \max\{T_0,
		\; \frac 1 {\alpha^2} \}.
	\end{align}
	Now define $\alpha_0=\ka_0<1$, and $\alpha_{k+1} =\alpha_k^2$. Note that
	\begin{align}
		\alpha_k =e^{2^k \ln \ka_0}.
	\end{align}
	Clearly it holds that
	\begin{align}
		\theta(t) \le \frac {\alpha_k} t, \qquad\forall\, t\ge \max\{T_0, \; \frac 1 {\alpha_k} \}.
	\end{align}
	Consider $t \in [\frac 1 {\alpha_k}, \; \frac 1 {\alpha_{k+1} } ]=[\frac 1 {\alpha_k},
	\frac 1 {\alpha_k^2}]$.  Clearly it holds that
	\begin{align}
		\alpha_k \le t^{-\frac 12}.
	\end{align}
	Thus we have for all $t\in [\frac 1 {\alpha_k}, \; \frac 1 {\alpha_{k+1} }]$ with
	$\frac 1{\alpha_k} \ge T_0$, it holds that
	\begin{align}
		\theta(t) \le t^{-\frac 32}.
	\end{align}

	Thus we have for all $t$ sufficiently large
	\begin{align}
		\theta(t) \le t^{-\frac 32}.
	\end{align}
	Iterating this estimate again we obtain $\theta(t) \le O(t^{-2})$.
\end{proof}

\begin{prop} \label{prTp8321}
	Assume $T\ge 3$. Suppose $\theta: \; [T,\infty)\to [0,\infty)$ is
	continuously differentiable and satisfy
	\begin{align}
		 \sup_{t\ge T_0} t \theta(t) <\infty;\qquad
		 \theta^{\prime}(t)=-\frac 32 \theta^2(t) +F(t), \quad \forall\, t>T,
	\end{align}
	where  for some $K_0>0$
	\begin{align}
		|F(t) | \le K_0 t^{-3}, \qquad\forall\, t\ge T.
	\end{align}
	Then there exists $\theta_* \in \mathbb R$, such that
	\begin{align}
		\theta(t) = \frac {\theta_*} t+ R(t),
	\end{align}
	where
	\begin{align}
		\sup_{t\ge T} \frac{|R(t)|} {\frac {\ln t} {t^2} } <\infty.
	\end{align}
\end{prop}
\begin{proof}
	Note that we only need to investigate the regime $t\gg 1$.
	We shall discuss two cases:
	
	\noindent Case 1. $\limsup\limits_{t\to \infty} t \theta(t) >0$. In this case we use Lemma
	\ref{leTp5212}.  Clearly for $T_0$ sufficiently large we have
	\begin{align}
		\theta(t) t \sim 1, \qquad\forall\, t\ge T_0.
	\end{align}
	From the ODE we obtain
	\begin{align}
		\frac d {dt} \Bigl( \frac 1 {\theta} \Bigr) = \frac 32 +O(t^{-1}).
	\end{align}
	It follows that for $T_0^{\prime}$ sufficiently large and all $t\ge T_0^{\prime}+2$,
	\begin{align}
		\theta(t) = \frac 1 {d_1 +d_2 (t-T_0^{\prime})+ O(\ln (t-T_0^{\prime} ))},
	\end{align}
	where $d_1>0$, $d_2>0$ are constants. The desired asymptotics then follows
	easily.
	
	\noindent  Case 2. $\limsup_{t\to \infty} \theta(t) t=0$. In this case we make a change of variable:
	\begin{align}
		t= N \tau, \; \theta(t) = \gamma \Theta(\tau), \quad \frac N {\gamma} F(t) =\tilde F(\tau).
	\end{align}
	Clearly
	\begin{align}
	\frac d {d\tau} \Theta(\tau) = -  \frac 32 \gamma N \Theta^2 + \tilde F(\tau),
	\quad\mbox{and}\quad
	|\tilde F(\tau)| \le K_0 \frac 1 {\gamma N^2} \tau^{-3}.
	\end{align}
	Thus if we take $\gamma = \frac 1{3N}$ and $N$ sufficiently large, we obtain
	\begin{align}
		\Theta(\tau) \le \frac 12 \int_{\tau}^{\infty} \Theta^2(s) ds + \frac 1 {2\tau^2},
		\qquad\forall\, \tau \ge \tau_0,
	\end{align}
	where $\tau_0$ is sufficiently large. We then use Lemma \ref{leTp5214}
	to conclude that $\Theta(\tau) =O(\tau^{-2})$. Thus in this case
	$\theta(t) = O(t^{-2})$.
\end{proof}

\section{Concluding remarks.}
In this work we considered the  classification of
steady states to the one-dimensional periodic Allen-Cahn equation with standard
double well potential. We gave a full classification of all possible steady states and identitified
their precise dependence on the diffusion coefficient in terms of energy and
profiles. We found a novel self-replicating property of steady state solutions
amongst the hierarchy
solutions organized according to the diffusion parameter. We developed
a new modulation theory around these steady states and proved sharp convergence
results. We discuss below a few possible future directions.

\begin{itemize}
	\item [(1).] Classification for other  models with different linear dissipations and
	 nonlinearities. Even for the classical Allen-Cahn case, one can investigate the singular
	 potential functions such as the logarithmic function given by
	\begin{align}f(u)=-u+\frac18\log\frac{1+u}{1-u},\end{align}
	or the sine-Gordon type
	\begin{align}f(u)=-\sin u.\end{align}
	Besides the one dimensional theory, we also expect  some generalizations to higher dimensions.
	
	\item [(2).] Patching and extension and solutions across general interfaces. It is natural to consider extending solutions of
	\begin{align}\mathcal Lu-f(u)=0\quad\mbox{in}\quad \Omega,\end{align}
	where $\mathcal L$ is the linear part and $f$ denotes the nonlinear part with appropriate boundary conditions. The task is to investigate under what conditions we can extend the solution across a portion of the boundary of $\Omega$ which is assumed to be a hypersurface or even some lower dimensional interface. In one dimension the situation is simple via reflection, but the general situation certainly merits further investigation.

	\item [(3).] Convergence theory for general initial data, and also for other equations and phase field models. These include nonlocal Allen-Cahn equations driven by general polynomials or logarithmic or even mildly singular nonlinearities, also one can investigate Cahn-Hillard equations, Molecular Beam epitaxy equation, time fractional equations and so on.
\end{itemize}

\appendix

\section{Proof of Proposition \ref{pra.1}}\label{append1}
We prove Proposition \ref{pra.1} step by step.

(1) In the case of $C>\frac12$, we can see that $u'$ never changes sign and it implies that $u$ is either an increasing or a decreasing function. In addition $|u'|$ has a positive lower bound, it implies $u$ is unbounded. Thus, there is no bounded solution.
\smallskip

(2) In the case of $C = \frac 1 2$, \eqref{a.id1} becomes
\begin{align}\ka^2(u')^2=\frac12(u^2-1)^2.\end{align}
It is easy to check that $u=1$ or $-1$ is always a solution to the above equation.
In the range $\{x\mid |u(x)|<1\}$ we solve the above ODE and get $u=\pm\tanh\frac{x+c}{\sqrt{2}\ka}$, it defines an entire solution of \eqref{a.ac}.
While if $\{x\mid |u(x)|>1\}$ is not empty, then we can use the ODE $u'=\frac {1}{\sqrt 2}(u^2-1)$ to derive that the solution must be unbounded.
Therefore, $\{x\mid |u(x)|>1\}=\emptyset$. As a result, we get either $|u|\equiv1$ or $u=\tanh\frac{x+c}{\sqrt{2}\ka}$ in this case.
\smallskip

(3) In the case of $0<C<\frac 1 2$, according to \eqref{a.id1}, we have
\begin{equation}
\frac 1 2(u^2-1)^2+C-\frac 1 2 \geq 0,
\end{equation}
i.e.,
\begin{equation}
|u|\geq\sqrt{1+\sqrt{1-2C}} \quad\mbox{or} \quad |u|\leq\sqrt{1-\sqrt{1-2C}}.
\end{equation}
If there exists some point $x_1$ such that $u(x_1)>\sqrt{1+\sqrt{1-2C}}$ (multiplying by $-1$ if necessary), then we see from \eqref{a.id1},  that either $u(x)$ is strictly increasing for $x\in(x_1,+\infty)$ and $u'(x)$ has a positive lower bound, or strictly decreasing for $x\in(-\infty,x_1)$ and $u'(x)$ has a negative upper bound.
Consequently, $u(x)$ is unbounded, which implies that such $x_1$ can not exist.
Further, $u \equiv \sqrt{1+\sqrt{1-2C}}$ is obviously not the solution to \eqref{a.ac}.
Therefore, we can claim that
\begin{equation}
\label{a.claim}
|u|\leq\sqrt{1-\sqrt{1-2C}}.
\end{equation}

Now we show that $u$ has local maxima and minima not at infinity.
Otherwise, $u(x)$ is monotonic  when $|x|$ is sufficiently large. Without loss of generality, suppose that $u(x)$ is monotone increasing for $x$ large. Then by \eqref{a.claim} and \eqref{a.id1} we have
\begin{align}\lim\limits_{x\to\infty}u(x)=\sqrt{1-\sqrt{1-2C}}\quad\mathrm{and}\quad \lim\limits_{x\to+\infty}u'(x)=0.\end{align}
Using \eqref{a.ac} we derive that $u''(x)\neq 0$ when $u(x)$ is around $\sqrt{1-\sqrt{1-2C}}$. Contradiction arises. Thus, $u(x)$ have both local maxima and minima.
Let $x=a$ be a local maximum point and $x=b$ be the closest local minimum point to $a$, then by \eqref{a.id1} we get
\begin{equation}
u(a)=\sqrt{1-\sqrt{1-2C}}\quad \mathrm{and}\quad
u(b)=-\sqrt{1-\sqrt{1-2C}}.
\end{equation}
By reflection symmetry,  $u(2b-a)=\sqrt{1-\sqrt{1-2C}}$. Repeating the reflection process we could see that $u(x)$ is a periodic function with minimal period equals to $|2b-2a|$.

(4) Finally, for the last statement, it can be obtained $|u|\geq \sqrt{2}$ or $u \equiv 0$ from \eqref{a.id1}.
From a discussion similar to the above, $|u|\geq \sqrt{2}$ will lead to contradiction.
Then we have $u\equiv 0$.

\vspace{1cm}
\begin{center}
	{\bf Acknowledgement}
\end{center}	

The research of T. Tang is partially supported by the Special Project on High-Performance Computing of the National Key R\&D Program under No. 2016YFB0200604, the National Natural Science Foundation of China (NSFC) Grant No. 11731006, the NSFC/Hong Kong RGC Joint Research Scheme (NSFC/RGC 11961160718), and the fund of the Guangdong Provincial Key Laboratory of Computational Science and Material Design (No. 2019B030301001).
The research of D. Li is supported in part by Hong Kong RGC grant GRF 16307317 and 16309518.
The research of W. Yang is supported by the National Natural Science Foundation of China (NSFC) Grant No. 11801550 and No. 11871470.
The research of C. Quan is supported by the National Natural Science Foundation of China (NSFC) Grant No. 11901281 and the Guangdong Basic and Applied Basic Research Foundation (2020A1515010336).

\vspace{1.5cm}

\frenchspacing
\bibliographystyle{plain}

\end{document}